\documentclass[final,12p,times,reqno]{amsart}

\usepackage[colorlinks=true,linkcolor=black, citecolor=blue, urlcolor=blue]{hyperref}
\usepackage{color}
\usepackage{amscd,amsthm,graphicx, color}
\usepackage{amsfonts,latexsym}
\usepackage{amssymb}
\usepackage{amsmath, times}
\usepackage{epsfig}
\usepackage{graphics}
\usepackage{latexsym,mathrsfs}
\newcommand{\supp}{{\textrm supp \ }}

\DeclareMathOperator{\supess}{esssup}
 
\author[B. Selmi]{Bilel Selmi}  

 \bigskip  \bigskip
\address{\newline Analysis, Probability $\&$ Fractals Laboratory LR18ES17, Department of Mathematics, Faculty of
Sciences of Monastir, University of Monastir, 5000-Monastir,
Tunisia}

\email{ bilel.selmi@fsm.rnu.tn} \email{ bilel.selmi@isetgb.rnu.tn}


\newcommand{\R}{\mathbb R}

\newcommand{\N}{\mathbb N}

\newtheorem{theorem}{Theorem}
\newtheorem{lemma}{Lemma}
\newtheorem{proposition}{Proposition}

\newtheorem{definition}{Definition}
\newtheorem{remark}{Remark}

\numberwithin{equation}{section}

\title[A review on multifractal analysis of Hewitt-Stromberg measures]{A review on multifractal analysis of Hewitt-Stromberg measures}

\begin{document}
 \maketitle
\begin{abstract}
We estimate the upper and lower bounds of the Hewitt-Stromberg
dimensions. In particular, these results give new proofs of theorems
on the multifractal formalism which is based on the Hewitt-Stromberg
measures and yield results even at points $q$ for which the upper
and lower multifractal Hewitt-Stromberg dimension functions differ.
Finally, concrete examples of a measure satisfying the above
property are developed.
\end{abstract}

\maketitle

                                                                  \section{Introduction}
The multifractal analysis is a natural framework to finely describe
geometrically the heterogeneity in the distribution at small scales
of the measures on a metric space. The multifractal formalism aims
at expressing the dimension of the level sets in terms of the
Legendre transform of some {\it free energy} function in analogy
with the usual thermodynamic theory.  One says that $\mu$ satisfies
the multifractal formalism if the Legendre transform of this free
energy yields the Hausdorff dimension of the level set of the local
H\"{o}lder exponent of $\mu$. For the measures we consider in this
article, the thermodynamic limit does not exist: the free energy
splits into two functions given by the upper and lower limits, and
the Legendre transforms of both of these functions have an
interpretation in terms of dimensions of the sets of
iso-singularities. However in the standard formalism discontinuities
of the free energy or one of its derivatives correspond to phase
transitions, we are facing a new phenomenon. It would be of interest
to know whether physical systems exhibiting such behavior exist.

The motivations of this paper come from several sources. The authors
in \cite{BBH} constructed "{\it bad}" measures whose Olsen's
multifractal functions $b_\mu$ and $B_\mu$ coincide at one or two
points only. These measures can fulfill the classical multifractal
formalism at one or two points only, i.e., the classical
multifractal formalism does not hold. Ben Nasr et al. in \cite{BBH}
give two constructions: The first one provides $b_\mu$ and $B_\mu$
functions with Lipschitz regularity. The second one provides real
analytic functions, but in this case, the support of the measure is
a Cantor set of Hausdorff dimensions less than $1$ (they use
inhomogeneous Bernoulli measures). Later, Ben Nasr et al. \cite{BJ}
and Shen \cite{SH1}, revisited the first example in \cite[Section
2.3]{BBH}, the meaning no interpretation of Olsen's multifractal
function was given in terms of dimensions. By the way, it was proven
in \cite{BJ} that, for some range of $\alpha$, the Hausdorff
dimension of the local H\"{o}lder exponent is given by the value of
the Legendre transform of $b_\mu$ at $\alpha$ and their packing
dimension is the value of the Legendre transform of $B_\mu$ at
$\alpha$ (This is the idea that we refine to get our results for the
multifractal formalism for the Hewitt-Stromberg measure). In
\cite{SH1},  Shen studies the main results of  \cite{BJ} for which
the function $b_\mu$ and $B_\mu$ can be real analytic functions. Let
us mention also that the authors in \cite{HLW, W3, W4, YZ} extended
these results to some Moran measures associated with homogeneous
Moran fractals. The above results were later generalized in
\cite{LiSe, SBSB33} to the relative multifractal formalism, and in
\cite{Samti} to inhomogeneous multinomial measures constructed on
the product symbolic space.

Motivated by the above papers, the authors in \cite{NB1, NB2, BD22,
SB, SB4} introduced and studied a new multifractal formalism based
on the Hewitt-Stromberg measures. We point out that this formalism
is completely parallel to Olsen's multifractal formalism introduced
in \cite{Ol1} which is based on the Hausdorff and packing measures.
In fact, the two most important (and well-known) measures in fractal
geometry are the Hausdorff measure and the packing measure. However,
in 1965, Hewitt and Stromberg introduced a further {\it fractal}
measure in their classical textbook \cite[Exercise (10.51)]{HeSt}.
Since then, these measures have  been investigated by several
authors, highlighting their importance in the study of local
properties of fractals and products of fractals. One can cite, for
example \cite{NB3, BD22, Ha1, Ha2, Ha3, Ha4, Ha5, JuMaMiOlSt, Olll2,
SB4, Zi}. In particular, Edgar's textbook \cite[pp. 32-36]{Ed}
provides an excellent and systematic introduction to these measures.
Such measures appears also explicitly, for example, in Pesin's
monograph \cite[5.3]{Pe} and implicitly in Mattila's text
\cite{Mat}. One of the purposes of this paper is to define and study
a class of natural multifractal analogue of the Hewitt-Stromberg
measures. While Hausdorff and packing measures are defined using
coverings and packings by families of sets with diameters less than
a given positive number $\delta$, say, the Hewitt-Stromberg measures
are defined using packings of balls with a fixed diameter $\delta$.

In the present paper we estimate the upper and lower bounds of the
Hewitt-Stromberg dimensions of a subset $E$ of $\mathbb{R}^n$. We
apply the main results to give new (and in several cases simpler)
proofs of theorems on the multifractal formalism which is based on
the Hewitt-Stromberg measures, that they yield results even at
points $q$ for which the the upper and lower multifractal
Hewitt-Stromberg dimension functions differ. We also give some
examples of a measure for which the multifractal functions are
different and for which the lower and upper multifractal
Hewitt-Stromberg functions are different and the lower and upper
Hewitt-Stromberg dimensions of the level sets of the local
H\"{o}lder exponent are given by the Legendre transform respectively
of lower and upper multifractal Hewitt-Stromberg dimension
functions.

The paper is structured as follows. In Section \ref{Section1} we
recall the definitions  of the various fractal and multifractal
dimensions and measures investigated in the paper. The definitions
of the Hausdorff and packing measures and the Hausdorff and packing
dimensions are recalled in Section \ref{sec2.1}, and the definitions
of the Hewitt-Stromberg measures are recalled in Section
\ref{sec2.2}, while the definitions of the Hausdorff and packing
measures are well-known, we have, nevertheless, decided to include
these-there are two main reasons for this: firstly, to make it
easier for the reader to compare and contrast the Hausdorff and
packing measures with the less well-known Hewitt-Stromberg measures,
and secondly, to provide a motivation for the Hewitt-Stromberg
measures. In Section \ref{sec2.4} we recall the definitions of the
multifractal Hewitt-Stromberg measures and separator functions, and
study their properties. In particular, this section recalls earlier
results on the values of the multifractal formalism based on
Hewitt-Stromberg measures developed in \cite{NB1, NB2, BD22}; this
discussion is included in order to motivate our main results
presented in Section \ref{sec3}. Section \ref{Estimates} gives some
estimates of the upper and lower bounds of the Hewitt-Stromberg
dimensions. In Sections \ref{Proofs}-\ref{differ} we apply the
results from Section \ref{Estimates} to give simple proofs of
theorems on the multifractal formalism developed in \cite{NB2} and
yield a result even at points $q$ for which the multifractal
Hewitt-Stromberg dimension functions differ. Finally, Section
\ref{sec4} contains concrete examples related to these concepts.

                                                               \section{Preliminaries}\label{Section1}

\subsection{Hausdorff measure, packing measure and dimensions}\label{sec2.1} While
the definitions of the Hausdorff and packing measures and the
Hausdorff and packing dimensions are well-known, we have,
nevertheless, decided to briefly recall the definitions below. There
are several reasons for this: firstly, since we are working in
general metric spaces, the different definitions that appear in the
literature may not all agree and for this reason it is useful to
state precisely the definitions that we are using; secondly, and
perhaps more importantly, the less well-known Hewitt-Stromberg
measures (see Section \ref{sec2.2}) play an important part in this
paper and to make it easier for the reader to compare and contrast
the definitions of the Hewitt-Stromberg measures and the definitions
of the Hausdorff and packing measures it is useful to recall the
definitions of the latter measures; and thirdly, in order to provide
a motivation for the Hewitt-Stromberg measures. Let $(X,d)$ be a
separable metric space, $E \subseteq X$ and $t>0$. Throughout this
paper, ${B}(x, r)$ stands for the open ball  $${B}(x, r)=\big\{y
\in{X} \big|\;\; \mathrm{d}(x, y)<r\big\}.$$

The Hausdorff measure is defined, for $\delta>0$, as follows
$$
\mathscr{H}_\delta^t(E)= \inf\left\{\sum_i
\Big(\text{diam}(E_i)\Big)^t\Big|\;\; E\subseteq\bigcup_i
E_i,\;\;\text{diam}(E_i)<\delta \right\}.
$$
This allows to define first the $t$-dimensional Hausdorff measure
$\mathscr{H}^t(E)$ of $E$ by
$$
\mathscr{H}^t(E)=\sup_{\delta>0}\mathscr{H}_\delta^t(E).
$$
Finally, the Hausdorff dimension $\dim_H(E)$ is defined by
$$
\dim_H(E)=\sup\Big\{t\geq0\;\big|\; \mathscr{H}^t(E)=+\infty\Big\}.
$$

The packing measure  is defined, for $\delta> 0$, as follows
\begin{eqnarray*}
\overline{\mathscr{P}}_\delta^t(E)= \sup\left\{\sum_i
\Big(2r_i\Big)^t\right\},
\end{eqnarray*}
where the supremum is taken over all open balls $\Big( B_i=B(x_i,
r_i) \Big)_i \; \text{such that}\; r_i \leq \delta\; \text{and
with}\; x_i \in E \; \text{and}\;B_i\cap B_j=\emptyset\; \text{for
all}\; i\neq j$. The $t$-dimensional packing pre-measure
$\overline{\mathscr{P}}^t(E)$ of $E$ is now defined by
$$
\overline{\mathscr{P}}^t(E)=\sup_{\delta>0}\overline{\mathscr{P}}_\delta^t(E).
$$
This makes us able to define the $t$-dimensional packing measure
${\mathscr{P}}^t(E)$ of $E$ as
$$
{\mathscr{P}}^t(E)=\inf\left\{\sum_i\overline{\mathscr{P}}^t(E_i)\;\Big|\;\;E\subseteq\bigcup_i
E_i\right\},
$$
and the packing dimension $\dim_P(E)$ is defined by
$$
\dim_P(E)=\sup\Big\{t\geq0\;\big|\; \mathscr{P}^t(E)=+\infty\Big\}.
$$
The reader is referred to Falconer's book  \cite{ FKJ} for an
excellent discussion of the Hausdorff measure and the packing
measure.

\subsection{Hewitt-Stromberg measures and dimensions}\label{sec2.2}
In this section, we recall the definitions of the Hewitt- Stromberg
measures, while the definitions of the Hausdorff and packing
measures are well-known, we have, nevertheless, decided to include
these-there are two main reasons for this: firstly, to make it
easier for the reader to compare and contrast the Hausdorff and
packing measures with the less well-known Hewitt-Stromberg measures,
and secondly, to provide a motivation for the Hewitt-Stromberg
measures. Let $X$ be a separable metric space and $E \subseteq X$.
For $t>0$, the Hewitt-Stromberg pre-measures are defined as follows,
$$
{\mathsf L}^t(E)=\liminf_{r\to0} N_r(E) \;(2r)^t\quad\text{and}\quad
\overline{\mathsf{H}}^t(E)=\sup_{F\subseteq E}{\mathsf L}^t(F),
$$
and
$$
\overline{\mathsf{P}}^t(E)=\limsup_{r\to0} M_r(E) \;(2r)^t,
$$
where the  covering number $N_r(E)$ of $E$ and the packing number
$M_r(E)$ of $E$ are given by
$$
N_r(E)=\inf\left\{\sharp\{I\}\;\Big|\; \Big( B(x_i, r) \Big)_{i\in
I} \; \text{is a family of open balls with}\; x_i \in E \;
\text{and}\; E\subseteq \bigcup _i B(x_i, r)\right\}
$$
and
$$
M_r(E)=\sup\left\{\sharp\{I\}\; \Big|\; \Big(B_i= B(x_i, r)
\Big)_{i\in I} \; \text{is a family of open balls with}\; x_i \in E
\; \text{and}\; B_i\cap B_j=\emptyset\; \text{for}\; i\neq
j\right\}.
$$

Now, we define the lower and upper $t$-dimensional Hewitt-Stromberg
measures,  which we denote respectively by ${\mathsf H}^t(E)$ and
${\mathsf P}^t(E)$,
 as follows
$$
{\mathsf{H}}^t(E)=\inf\left\{\sum_i\overline{\mathsf{H}}^t(E_i)\;\Big|\;\;E\subseteq\bigcup_i
E_i\right\}
$$
and
$$
{\mathsf{P}}^t(E)=\inf\left\{\sum_i\overline{\mathsf{P}}^t(E_i)\;\Big|\;\;E\subseteq\bigcup_i
E_i\right\}.
$$

We recall some basic inequalities satisfied by the Hewitt-Stromberg,
the Hausdorff  and the packing measure
$$
\overline{\mathsf{H}}^t(E)\leq \xi\overline{\mathsf{P}}^t(E)\leq \xi
\overline{\mathscr{P}}^t(E)
$$
and
$$
{\mathscr{H}}^t(E)\leq{\mathsf{H}}^t(E)\leq\xi {\mathsf{P}}^t(E)\leq
\xi {\mathscr{P}}^t(E),
$$
where $\xi$ is the constant that appears in Besicovitch's covering
theorem.\\
The reader is referred to Edgar's book \cite[pp.~32]{Ed} (see also
\cite[Proposition 2.1]{JuMaMiOlSt, Olll2}) for a systematic
introduction to the Hewitt-Stromberg measures.

\bigskip
The lower and upper Hewitt-Stromberg dimension
$\underline{\dim}_{MB}(E)$ and  $\overline{\dim}_{MB}(E)$ are
defined by
$$
\underline{\dim}_{MB}(E)=\inf\Big\{t\geq0\;\Big|\;
{\mathsf{H}}^t(E)=0\Big\}=\sup\Big\{t\geq0\;\Big|\;
{\mathsf{H}}^t(E)=+\infty\Big\}
$$
and
$$
\overline{\dim}_{MB}(E)=\inf\Big\{t\geq0\;\big|\;
{\mathsf{P}}^t(E)=0\Big\}=\sup\Big\{t\geq0\;\big|\;
{\mathsf{P}}^t(E)=+\infty\Big\}.
$$
The lower and upper box dimensions, denoted by
$\underline{\dim}_{B}(E)$ and $\overline{\dim}_{B}(E)$,
respectively, are now defined by
$$
\underline{\dim}_{B}(E)=\liminf_{r\to0}\frac{\log N_r(E)}{-\log
r}=\liminf_{r\to0}\frac{\log M_r(E)}{-\log r}
$$
and
$$
\overline{\dim}_{B}(E)=\limsup_{r\to0}\frac{\log N_r(E)}{-\log
r}=\limsup_{r\to0}\frac{\log M_r(E)}{-\log r}.
$$
These dimensions satisfy the following inequalities,
$$
\dim_H(E)\leq\underline{\dim}_{MB}(E)\leq\overline{\dim}_{MB}(E)\leq\dim_P(E),
$$

$$
\dim_H(E)\leq \dim_P(E)\leq\overline{\dim}_{B}(E)
$$
and
$$\dim_H(E)\leq \underline{\dim}_{B}(E)\leq\overline{\dim}_{B}(E).
$$
In particular, we have (see \cite{FKJ, MiOl})
$$
\underline{\dim}_{MB}(E)= \inf \left\{ \sup_{i}
\underline{\dim}_{B}(E_i)\;\Big|\;E\subseteq \bigcup_i E_i, \;\; E_i
\;\; \text{are bounded in } \; X\right\}
$$
and
$$
\overline{\dim}_{MB}(E)= \inf \left\{ \sup_{i}
\overline{\dim}_{B}(E_i) \;\Big|\; E\subseteq \bigcup_i E_i, \;\;
E_i \;\; \text{are bounded in } \; X\right\}.
$$
The reader is referred to \cite{Ed, FKJ} for an excellent discussion
of the Hausdorff dimension, the packing dimension, lower and upper
Hewitt-Stromberg dimension and the box dimensions.

\subsection{Multifractal Hewitt-Stromberg measures and
dimension functions}\label{sec2.4}

This section gives a brief summary of the main results in \cite{NB1,
NB2}. We recall the definitions  of the  lower and upper
multifractal Hewitt-Stromberg measures and dimension functions. Let
$q, t \in \R$ and $\mu$ be a compactly supported Borel probability
measure on $\mathbb{R}^n$. For $E\subseteq {K}=:\supp \mu$, the
pre-measure of $E$ is defined by
$$
\overline{ {\mathsf P}}_\mu^{q,t}(E)= \limsup_{r\to0} M_{\mu,r}^q(E)
(2r)^t,
$$
where
$$
M_{\mu,r}^q(E)=\sup \left\{\displaystyle \sum_i
\mu(B(x_i,r))^q\;\Big|\;\Big(B(x_i,r)\Big)_i\;\text{is a centered
packing of}\; E \right\}.
$$
It is clear that  $\overline{{\mathsf P}}_\mu^{q,t}$ is increasing
and $\overline{{\mathsf P}}_\mu^{q,t}(\emptyset ) =0$. However it is
not $\sigma$-additive. For this, by using the standard Method I
construction \cite[Theorem 4]{RCA}, we introduce  the
${\mathsf{P}}_\mu^{q,t}$-measure defined by
$$
{\mathsf{P}}_\mu^{q,t}(E)=\inf \left\{\displaystyle
\sum_i{\overline{\mathsf P}}_\mu^{q,t}(E_i)\;\Big|\;
E\subseteq\bigcup_i E_i\;\right\}.
 $$
In a similar way we define
$$
{\mathsf L}_\mu^{q,t}(E)= \liminf_{r\to0} N_{\mu,r}^q(E) (2r)^t,
$$
where
$$
N_{\mu,r}^q(E)=\inf \left\{\displaystyle \sum_i
\mu(B(x_i,r))^q\;\Big|\;\Big(B(x_i,r)\Big)_i\;\text{is a centered
covering of}\; E \right\}.
$$
Since ${\mathsf L}_\mu^{q,t}$  is  not increasing and not countably
sub-additive, one needs a standard modification to get an outer
measure. Hence, we modify the definition as follows
 $$
\overline{\mathsf{H}}_\mu^{q,t}(E)=\sup_{F\subseteq E}
{\mathsf{L}}_\mu^{q,t}(F)
 $$
and by applying now the standard Method I construction \cite[Theorem
4]{RCA}, we obtain
$$
\mathsf{H}_\mu^{q,t}(E)=\inf \left\{\displaystyle
\sum_i\overline{\mathsf H}_\mu^{q,t}(E_i)\;\Big|\;
E\subseteq\bigcup_i E_i\; \right\}.
$$

 \bigskip \bigskip
The measure $\mathsf{ H}^{q,t}_{\mu}$ is of course a multifractal
generalization of the lower $t$-dimensional Hewitt-Stromberg measure
${\mathsf{H}}^t$, whereas $\mathsf{ P}^{q,t}_{\mu}$ is a
multifractal generalization of the upper $t$-dimensional
Hewitt-Stromberg measures ${\mathsf{P}}^t$. In fact, it is easily
seen that, for $t>0$,  one has
$$
\mathsf{ H}^{0,t}_{\mu}={\mathsf{H}}^t\quad\text{and}\quad\mathsf{
P}^{0,t}_{\mu}={\mathsf{P}}^t.
$$

The following result describes some of the basic properties of the
multifractal Hewitt-Stromberg measures including the fact that
${\mathsf H}_\mu^{q,t}$ and ${\mathsf P}_\mu^{q,t}$ are outer
measures and summarizes the basic inequalities satisfied by these
measures.
\begin{theorem}\label{HHPP}\cite{NB1}
Let $q, t \in \R$. Then for every set $E\subseteq {K}$ we have
\begin{enumerate}
\item the set functions $\mathsf{H}_\mu^{q,t}$ and $\mathsf{P}_\mu^{q,t}$ are outer measures and
thus they are measures on the algebra of the measurable sets.
\item There exists an integer $\xi\in\mathbb{N}$, such that $\mathsf{H}_\mu^{q,t}(E)\leq \xi
\mathsf{P}_\mu^{q,t}(E).$
\end{enumerate}
\end{theorem}

The measures $\mathsf{ H}^{q,t}_{\mu}$ and $\mathsf{ P}^{q,t}_{\mu}$
and the pre-measure $\overline{\mathsf{P}}^{q,t}_{\mu}$ assign in
the usual way a multifractal dimension to each subset $E$ of
$\mathbb{R}^n$. They are respectively denoted by
$\mathsf{b}_{\mu}^q(E)$, $\mathsf{B}_{\mu}^q(E)$ and
$\mathsf{\Delta}_{\mu}^q(E)$,
\begin{proposition}\cite{NB2} Let $q \in \R$ and $E\subseteq {K}$. Then
\par\noindent\begin{enumerate}
\item  there exists a unique number ${\mathsf  b}_{\mu}^q(E)\in[-\infty,+\infty]$ such that
 $$
\mathsf{H}^{q,t}_{\mu}(E)=\left\{\begin{matrix}  \infty &\text{if}& t < {\mathsf  b}_{\mu}^q(E),\\
 \\
 0 & \text{if}&  {\mathsf  b}_{\mu}^q(E) < t,\end{matrix}\right.
 $$
\item  there exists a unique number ${\mathsf  B}_{\mu}^q(E)\in[-\infty,+\infty]$ such that
 $$
\mathsf{P}^{q,t}_{\mu}(E)=\left\{\begin{matrix}  \infty &\text{if}& t < {\mathsf  B}_{\mu}^q(E),\\
 \\
 0 & \text{if}&  {\mathsf  B}_{\mu}^q(E) < t,\end{matrix}\right.
 $$

\item  there exists a unique number ${\mathsf  \Delta}_{\mu}^q(E)\in[-\infty,+\infty]$ such that
 $$
\overline{\mathsf{P}}^{q,t}_{\mu}(E)=\left\{\begin{matrix}  \infty &\text{if}& t < {\mathsf  \Delta}_{\mu}^q(E),\\
 \\
 0 & \text{if}&  {\mathsf  \Delta }_{\mu}^q(E) < t.\end{matrix}\right.
 $$
\end{enumerate}
In addition, we have
$${\mathsf  b}_{\mu}^q(E) \le   {\mathsf  B }_{\mu}^q(E) \le  {\mathsf  \Delta}_{\mu}^q(E).$$
\end{proposition}
The number ${\mathsf  b}_{\mu}^q(E)$ is an obvious multifractal
analogue of the lower Hewitt-Stromberg dimension
$\underline{\dim}_{MB}(E)$ of $E$ whereas ${\mathsf  B}_{\mu}^q(E)$
is an obvious multifractal analogues of the upper Hewitt-Stromberg
dimension $\overline{\dim}_{MB}(E)$ of $E$. In fact, it follows
immediately from the definitions that
$$
{\mathsf
b}_{\mu}^0(E)=\underline{\dim}_{MB}(E)\quad\text{and}\quad{\mathsf
B}_{\mu}^0(E)=\overline{\dim}_{MB}(E).
$$
Next, we define the separator functions ${\mathsf  \Delta}_{\mu}$,
${\mathsf  B}_{\mu}$ and ${\mathsf  b}_{\mu}$ :
$\mathbb{R}\rightarrow [-\infty, +\infty]$ by,
\begin{center}
${\mathsf  \Delta}_{\mu}(q)={\mathsf  \Delta}_{\mu}^q(K)$,\;\;\;
${\mathsf B}_{\mu}(q)={\mathsf  B}_{\mu}^q(K)$\;\; and\;\;\;
${\mathsf b}_{\mu}(q)={\mathsf  b}_{\mu}^q(K).$
\end{center}
The definition of these dimension functions makes it clear that they
are counterparts of the $\tau$-function which appears in the
multifractal formalism. This being the case, it is important that
they have the properties described by the physicists. The next
proposition shows that these functions do indeed have some of these
properties.
\begin{proposition}\cite{NB2}\label{Prop1}
Let $q\in\mathbb{R}$ and $E\subseteq {K}$.
\begin{enumerate}
\item The functions $q \mapsto \mathsf{ H}^{q,t}_{\mu}(E)$,  $\mathsf{P}^{q,t}_{\mu}(E)$,  $ \mathsf{C}^{q,t}_{\mu}(E)$ are decreasing.
\\
\item The functions $t \mapsto \mathsf{ H}^{q,t}_{\mu}(E)$,  $\mathsf{P}^{q,t}_{\mu}(E)$,  $ \mathsf{C}^{q,t}_{\mu}(E)$ are
decreasing.\\
\item The functions $q \mapsto { {\mathsf  b}}^{q}_{\mu}(E)$, $ {{\mathsf  B}}^{q}_{\mu}(E)$,  $ {{\mathsf  \Delta}}^{q}_{\mu}(E)$
are decreasing.\\
\item The functions $q \mapsto  {{\mathsf  B}}^{q}_{\mu}(E)$,  $ {{\mathsf  \Delta}}^{q}_{\mu}(E)$ are
convex.\\
\item For $q < 1$, we have  $0\leq{{\mathsf b}}_{\mu}(q)\leq {{\mathsf B}}_{\mu}(q)\leq {{\mathsf \Delta}}_{\mu}(q)
$.\\
\item For $q = 1$, we have  ${{\mathsf b}}_{\mu}(q)= {{\mathsf B}}_{\mu}(q)= {{\mathsf \Delta}}_{\mu}(q)  =
0$.\\
\item For $q > 1$, we have ${{\mathsf b}}_{\mu}(q)\leq {{\mathsf B}}_{\mu}(q)\leq {{\mathsf \Delta}}_{\mu}(q) \le 0$.
\end{enumerate}
\end{proposition}
\begin{remark}
It is easy to check that we get the same values for fractal and
multifractal measures and dimensions if we use just open balls or
just closed balls and would not change the results (for more detail
see \cite{BJ, FKJ, RTr, TR}). Note that the use of just closed balls
is necessary where the use of open balls was erroneous when applying
Vitali's theorem and in this case, we need to apply the {
Besicovitch covering theorem}. In particular, if we use closed
balls, then we obtain
$$
\overline{\mathsf{H}}^t(E)\leq \overline{\mathsf{P}}^t(E)\leq
\overline{\mathscr{P}}^t(E)
$$
and
$$
{\mathscr{H}}^t(E)\leq{\mathsf{H}}^t(E)\leq {\mathsf{P}}^t(E)\leq
{\mathscr{P}}^t(E).
$$
In addition, when $q \le 0$ or $q>0$ and $\mu$ satisfies the
doubling condition, we have $$\mathsf{H}_\mu^{q,t}(E)\leq
\mathsf{P}_\mu^{q,t}(E).$$
\end{remark}

 \bigskip \bigskip
The upper and lower local dimensions of a measure $\mu$ on
$\mathbb{R}^{n}$ at a point $x$ are respectively given by
$$
\overline{\alpha}_{\mu}(x)=\limsup _{r \rightarrow 0} \frac{\log
\mu(B(x, r))}{\log r}
$$
and
$$
\underline{\alpha}_{\mu}(x)=\liminf _{r \rightarrow 0} \frac{\log
\mu(B(x, r))}{\log r},
$$
where $B(x, r)$ denote the open ball of center $x$ and radius $r.$
We refer to the common value as the local dimension of $\mu$ at $x$,
and denote it by $\alpha_{\mu}(x)$. For $\alpha, \beta \geq 0$, let
us introduce the fractal sets
$$
\overline{E}_{\mu}(\beta)=\Big\{x \in K \mid
\overline{\alpha}_{\mu}(x) \leq \beta\Big\},
$$
$$
\underline{E}_{\mu}(\alpha)=\Big\{x \in K \mid
\underline{\alpha}_{\mu}(x) \geq \alpha\Big\},
$$
$$
E_{\mu}(\alpha, \beta)=\underline{E}_{\mu}(\alpha) \cap
\overline{E}_{\mu}(\beta)$$ and
$$
E_{\mu}(\alpha)=\underline{E}_{\mu}(\alpha)
\cap \overline{E}_{\mu}(\alpha).
$$
Before stating this formally, we remind the reader that if $\chi:
\mathbb{R} \rightarrow \mathbb{R}$ is a real valued function, then
the Legendre transform $\chi^{*}: \mathbb{R}
\rightarrow[-\infty,+\infty]$ of $\chi$ is defined by
$$
\chi^{*}(x)=\inf _{y}(x y+\chi(y)) .
$$

The multifractal formalism based on the measures ${\mathsf
H}^{q,t}_{\mu}$ and ${\mathsf P}^{q,t}_{\mu}$ and the dimension
functions $\mathsf{b}_{\mu}$, $\mathsf{B}_{\mu}$ and
$\mathsf{\Delta}_{\mu}$ provides a natural, unifying and very
general multifractal theory which includes all the hitherto
introduced multifractal parameters, i.e., the multifractal spectra
functions
$$
\alpha\mapsto
\mathsf{f}_\mu(\alpha)=:\underline{\dim}_{MB}E_\mu(\alpha)
$$
and
$$
\alpha\mapsto
\mathsf{F}_\mu(\alpha)=:\overline{\dim}_{MB}E_\mu(\alpha).
$$

The dimension functions $\mathsf{b}_{\mu}$ and $\mathsf{B}_{\mu}$
are intimately related to the spectra functions $\mathsf{f}_\mu$ and
$\mathsf{F}_\mu$, whereas the dimension function
$\mathsf{\Delta}_{\mu}$ is closely related to the upper box spectrum
(more precisely, to the upper multifractal box dimension function).
Let us briefly recall the notations and the main results proved in
\cite{NB2}. We say that the multifractal formalism which is based on
the Hewitt-Stromberg measures holds if,
$$
\underline{\dim}_{MB}(E_\mu(\alpha))=\overline{\dim}_{MB}(E_\mu(\alpha))=
{\mathsf b}_\mu^*(\alpha)={\mathsf B}_\mu^*(\alpha).
$$
One important thing which should be noted is that there are many
measures for which the multifractal formalism does not hold. In
fact, one question which several measure theorists are interested in
is, can we find a necessary and sufficient condition for the
multifractal formalism to hold? The authors in \cite{NB2} proved the
following statement.
\begin{theorem}\label{Jcole}
Let $\mu$ be a compactly supported Borel probability measure on
$\mathbb{R}^n$. Define
$$\underline{\alpha}=\displaystyle\sup_{0<q} -\frac{{{
\mathsf b}_{\mu}(q)}}{q}\quad\text{and}\quad
\overline{\alpha}=\displaystyle\inf_{0>q} -\frac{{{ \mathsf
b}_{\mu}(q)}}{q}.$$ Then,
$$
\underline{\dim}_{MB} \Big(E_{\mu}(\alpha)\Big)\leq { \mathsf
b}_{\mu}^*(\alpha)\quad\text{and}\quad \overline{\dim}_{MB}
\Big(E_{\mu}(\alpha)\Big)\leq {\mathsf
B}_{\mu}^*(\alpha)\quad\text{for all}\quad \alpha\in
(\underline{\alpha}, \;\overline{\alpha}).
$$
\end{theorem}
It is more difficult to obtain a minoration for the dimensions of
the sets described in Theorem \ref{Jcole}. Selmi et al. \cite{NB2}
gave a sufficient condition for a valid multifractal formalism as
follows.
\begin{theorem}\label{ah1}
Let $q\in\mathbb{R}$ and suppose that ${\mathsf
H}^{q,\mathsf{B}_{\mu}(q)}_{\mu}\big( K\big)>0.$ Then,
$$
\underline{\dim}_{MB}\Big(E_{\mu}\big(-\mathsf{B}_{\mu_+}'(q),-\mathsf{B}_{\mu_-}'(q)\big)\Big)\geq
\left\{
  \begin{array}{ll}
  -q\mathsf{B}_{\mu_+}'(q)+\mathsf{B}_{\mu}(q), & \hbox{if } q\geq0 \\\\
-q\mathsf{B}_{\mu_-}'(q)+\mathsf{B}_{\mu}(q), & \hbox{if } q\leq0.
  \end{array}
\right.
$$
In particular, if $\mathsf{B}_{\mu}$ is differentiable at $q$, one
has
 $$
\underline{\dim}_{MB}\Big(E_{\mu}(-\mathsf{B}'_{\mu}(q))\Big)=\overline{\dim}_{MB}
\Big(E_{\mu}(-\mathsf{B}'_{\mu}(q))\Big)=
\mathsf{B}^*_{\mu}\big(-\mathsf{B}'_{\mu}(q)\big)=\mathsf{b}^*_{\mu}\big(-\mathsf{B}'_{\mu}(q)\big).
 $$
Conversely, if
$\underline{\dim}_{MB}\Big(E_{\mu}(-\mathsf{B}'_{\mu}(q))\Big)\geq\mathsf{B}^*_{\mu}\big(-\mathsf{B}'_{\mu}(q)\big),$
then $\mathsf{b}(q)=\mathsf{B}(q).$
\end{theorem}
From the last part, when $\mathsf{B}'_{\mu}(q)$ exists,
$\mathsf{b}(q)=\mathsf{B}(q)$, known as an analogue of Taylor
regularity condition, is the necessary condition for a valid
multifractal formalism. Nevertheless, we don't know if the weaker
condition $\mathsf{b}(q)=\mathsf{B}(q)$ is sufficient to obtain the
conclusion of the first part of Theorem \ref{ah1}.

\section{Results and new proofs}\label{sec3}
In this section we present our main results: we estimate the upper
and lower bounds of the Hewitt-Stromberg dimensions of a subset $E$
of $\mathbb{R}^n$. We apply these results to give a new (and in
several cases simpler) proof of Theorems \ref{Jcole} and \ref{ah1}
and provide some results even at points $q$ for which the
multifractal dimension functions $\mathsf{b}_{\mu}(q)$ and
$\mathsf{B}_{\mu}(q)$ differ.

\subsection{Estimates for the Hewitt-Stromberg
dimensions}\label{Estimates} Let $\nu$ be a compactly supported
Borel probability measure on $\mathbb{R}^n$ with
$\mathscr{K}=:\supp\nu$. Throughout this paper, we denote for any
$E\subseteq \mathscr{K}$, $$\pi({E})={{\mathsf
H}_{\nu}^{1,0}}({E}).$$
\begin{theorem}\label{thinf}
Let $E$ be a bounded subset of $\mathbb{R}^n$ and suppose that
${\mathsf B}_{\nu}^{1}(E)\leq 0$,
then\\
\begin{eqnarray*}\label{th1}
\underline{\dim}_{MB}(E)\leq \sup_{x\in E}
\underline{\alpha}_{\nu}(x) \quad\text{and}\quad
\overline{\dim}_{MB}(E)\leq \sup_{x\in E}
\overline{\alpha}_{\nu}(x).
\end{eqnarray*}
\end{theorem}
\begin{proof}
Take $\alpha>\sup_{x\in E} \underline{\alpha}_{\nu}(x)$ and
$\varepsilon>0$. It follows from  ${\mathsf B}_{\nu}^{1}(E)\leq 0$
that ${\mathsf P}^{1,\varepsilon}_{\nu}(E)=0$, then we can choose a
sequence $(E_k)_k$ such that $E=\bigcup_k E_k$,
$$
\sum_k \overline{\mathsf P}^{1,\frac\varepsilon
2}_{\nu}(E_k)<1\quad\text{and}\quad \sum_k \overline{\mathsf
P}^{1,\varepsilon}_{\nu}(E_k)=0.
$$
Fix $k\in \mathbb{N}$. Let $F\subset E_k$ and $\delta>0$. For all
$x\in F$ we can find $\lambda_x \ge 2$ and
$\frac{\delta}{\lambda_x}<r_x<\delta,$ such that
 $$
    \nu(B(x,r_x))> r_x^\alpha.
 $$
The family $\Big(B(x,r_x)\Big)_{x\in F}$ is a centered
$\delta$-covering of $ \overline{F}.$ Then, we can choose a finite
subset $J$ of $\N$ such that the family
$\Big(B(x_i,r_{x_i})\Big)_{i\in J}$ is a centered $\delta$-covering
of $F$. Take $\lambda = \max\{\lambda_{x_i};\;{i\in J}\}$, then for
all $i\in J$, we have
$$
 \nu(B(x_i, \delta)) \ge  \nu(B(x_i,r_{x_i}))> r_{x_i}^\alpha \ge \left(\frac{\delta}{\lambda}\right)^\alpha.
$$
Since $\Big(B(x_i , \delta )\Big)_{ i\in J}$  is a centered covering
of $ F,$ then by using Besicovitch's covering theorem, we can
construct $\xi_n$ finite  sub-families $\Big(B(x_{1j},
\delta)\Big)_j$, \ldots ,$\Big(B(x_{\xi_n j}, \delta )\Big)_j$, such
that each  $$
F\subseteq\displaystyle\bigcup_{i=1}^{\xi_n}\bigcup_jB(x_{ij},
\delta)$$ and $\Big(B_{ij}=B(x_{ij},\delta)\Big)_j$ is a packing of
$F$ and such that
$\nu\big(B_{ij}\big)>\left(\frac{\delta}{\lambda}\right)^\alpha$.
Observing that
$$
N_\delta(F) (2\delta)^{\alpha+\varepsilon}\leq\sum_{i,j}
(2\delta)^{\alpha+\varepsilon}\leq (2\lambda)^\alpha\sum_{i,j}
\nu\big(B_{ij}\big)(2\delta)^{\varepsilon}\leq
(2\lambda)^\alpha\xi_n M_{\nu,\delta}^1(F)(2\delta)^{\varepsilon}.
$$
It follows from this that $$
{\mathsf{L}}^{\alpha+\varepsilon}(F)\leq (2\lambda)^\alpha\xi_n
\overline{\mathsf P}^{1,\varepsilon}_{\nu}(F)\leq
(2\lambda)^\alpha\xi_n \overline{\mathsf
P}^{1,\varepsilon}_{\nu}(E_k)
$$
which implies that
$$
\overline{\mathsf{H}}^{\alpha+\varepsilon}(E_k)\leq
(2\lambda)^\alpha\xi_n \overline{\mathsf
P}^{1,\varepsilon}_{\nu}(E_k)$$
and
$$
{\mathsf{H}}^{\alpha+\varepsilon}(E)\leq \sum_k
\overline{\mathsf{H}}^{\alpha+\varepsilon}(E_k) \leq
(2\lambda)^\alpha\xi_n\sum_k \overline{\mathsf
P}^{1,\varepsilon}_{\nu}(E_k)=0.
$$
We therefore conclude that
$$
\underline{\dim}_{MB}(E)\leq\alpha+\varepsilon \quad\text{for
all}\quad \varepsilon>0.$$ Finally, we get
$$
\underline{\dim}_{MB}(E)\leq \sup_{x\in E}
\underline{\alpha}_{\nu}(x).
$$

 \bigskip
Now, take $\alpha>\sup_{x\in E} \overline{\alpha}_{\nu}(x)$ and
$\varepsilon>0$. Since ${\mathsf B}_{\nu}^{1}(E)\leq 0$, we have
${\mathsf P}^{1,\varepsilon}_{\nu}(E)=0$, then there exists
$(E_j)_j$ such that $E=\bigcup_j E_j$,
$$
\sum_j \overline{\mathsf P}^{1,\frac\varepsilon
2}_{\nu}(E_j)<1\quad\text{and}\quad \sum_j \overline{\mathsf
P}^{1,\varepsilon}_{\nu}(E_j)=0.
$$
For all $x\in E$, we can therefore choose $\delta>0$ such that, for
all $0<r<\delta$, we have
$$
\nu\big(B(x,r)\big)>r^\alpha.
$$
Put the set
$$
E_m=\left\{x\in E\;\Biggl|\;\; \text{for all}\; r<
\frac1m,\;\;\nu\big(B(x,r)\big)>r^\alpha\right\}.
$$
Fix $m\in \mathbb{N}$ and $0<r<\min(\delta,\frac1m)$. Let
$\Big(B_i=B(x_i,r)\Big)_i$ be a packing of $E_j\cap E_m$. Then
$$
M_r(E_j\cap E_m) (2r)^{\alpha+\varepsilon} \leq2^\alpha \sum_i
\nu\big(B_i\big) (2r)^{\varepsilon} \leq 2^\alpha
M_{\nu,r}^1(E_j)(2r)^{\varepsilon}
$$
and
$$
\overline{\mathsf P}^{\alpha+\varepsilon}(E_j \cap E_m)\leq
2^\alpha\overline{\mathsf P}^{1,\varepsilon}_{\nu}(E_j)=0.
$$
This clearly implies that
$$
{\mathsf P}^{\alpha+\varepsilon}(E_m)\leq \sum_j\overline{\mathsf
P}^{\alpha+\varepsilon}(E_j \cap E_m)\leq
2^\alpha\sum_j\overline{\mathsf P}^{1,\varepsilon}_{\nu}(E_j)=0.
$$
We deduce the result from the fact that $E_m \nearrow E$. This
completes the proof of Theorem \ref{thinf}.
\end{proof}

It always needs an extra condition to obtain a lower bound for the
dimensions of sets.
\begin{theorem}\label{thess}
Let $E\subseteq \mathscr{K}$ and assume that $\pi(E)> 0$, then\\
\begin{eqnarray*}\label{th3}
\underline{\dim}_{MB}(E)\geq{ _{\substack{\displaystyle\supess\\
_{x\in E}}}} \underline{\alpha}_{\nu}(x) \quad\text{and}\quad
\overline{\dim}_{MB}(E)\geq { _{\substack{\displaystyle\supess\\
_{x\in E}}}}\overline{\alpha}_{\nu}(x),
\end{eqnarray*}
where the essential bounds being related to the measure $ \pi$.
\end{theorem}
\begin{proof}
Let $\alpha<{ _{\substack{\displaystyle\supess\\ _{x\in
E}}}}\underline{\alpha}_{\nu}(x)$. Consider the set
$$
F=\left\{x\in E\;\Biggl|\;\; \liminf_{r\to 0}\frac{\log \nu\big(B(x,
r)\big)}{\log r}>\alpha\right\}.
$$
It is clear that $\pi(F)>0$. For all $x\in E$ we can find $\delta>0$
such that for all $0<r<\delta$, we have $$\nu\big(B(x,r)\big)\leq
r^\alpha.$$ Now, let $(F_j)_j$ be a countable partition  of $F$. We
put forward the set
$$
F_{j_p}=\left\{x\in F_j\;\Biggl|\;\;r<\frac1p,\;\;
\nu\big(B(x,r)\big)\leq r^\alpha\right\}.
$$
Fix $p\in \mathbb{N}$ and  $G$ be a subset of $F_{j_p}$. Let
$0<r<\min(\delta,\frac1p)$ and
$\Big(B_i=B(x_i,r)\Big)_{i\in\{1,...,N_r(G)\}}$ be a centered
covering of $G$, then
$$
N_{\nu,r}^{1}(G)(2r)^0\leq \sum_i \nu(B_i)\leq 2^{-\alpha} N_r(G)
(2r)^\alpha.
$$
This clearly implies that
$${\mathsf
L}_{\nu}^{1,0}(G)\leq2^{-\alpha}{{\mathsf
L}}^{\alpha}(G)\leq2^{-\alpha}\overline{{\mathsf
H}}^{\alpha}(G)\leq2^{-\alpha} \overline{\mathsf
H}^{\alpha}(F_{j_p})
$$
and
$$
\overline{\mathsf
H}_{\nu}^{1,0}(F_{j_p})\leq2^{-\alpha} \overline{\mathsf
H}^{\alpha}(F_{j_p}).
$$
Since $F_j=\bigcup_p F_{j_p}$ for all $j$ and $\pi\big(F\big)>0$, by
making $\delta\to 0$, we obtain
$$
0<{\mathsf H}_{\nu}^{1,0}(F)\leq\sum_j\sum_p\overline{\mathsf
H}_{\nu}^{1,0}(F_{j_p})\leq2^{-\alpha}\sum_j\sum_p\overline{\mathsf
H}^{\alpha}(F_{j_p})
$$
which implies that
$$
0<{\mathsf H}^{\alpha}(F)\leq{\mathsf H}^{\alpha}(E).
$$
 We therefore conclude that
$$
\underline{\dim}_{MB}(E)\geq \alpha\quad \text{for all}\quad
\alpha<{ _{\substack{\displaystyle\supess\\ _{x\in E}}}}
\underline{\alpha}_{\nu}(x).
$$

 \bigskip
Let $\alpha<{ _{\substack{\displaystyle\supess\\ _{x\in
E}}}}\overline{\alpha}_{\nu}(x)$. Consider the set
$$
F=\left\{x\in E\;\Biggl|\;\; \limsup_{r\to 0}\frac{\log
\nu\big(B(x,r)\big)}{\log r}>\alpha\right\},
$$
then $\pi(F)>0$. Fix $G\subset F$, then for all $x\in G$ and all
$\delta>0$, we can find a positive real number $0<r_x<\delta$ such
that
$$
\nu\big(B(x,r_x)\big)\leq r_x^\alpha.
$$
The family $\Big(B(x,r_x)\Big)_{x\in G}$ is a centered
$\delta$-covering of $ \overline{G}.$ Then, we can choose a finite
subset $J$ of $\N$ such that the family
$\Big(B(x_i,r_{x_i})\Big)_{i\in J}$ is a centered $\delta$-covering
of $G$. Take $r = \max\{ r_{x_i},\;{i\in J}\}$, then $\Big(B(x_i ,
r)\Big)_{ i\in J}$  is a centered covering of $G.$ By using
Besicovitch's covering theorem, we can construct $\xi_n$ finite
sub-families $\Big(B(x_{1j},r)\Big)_j$,....,$\Big(B(x_{\xi_n
j},r)\Big)_j$ such that each  $$
G\subseteq\displaystyle\bigcup_{i=1}^{\xi_n}\bigcup_jB(x_{ij},r),$$
$\Big(B_{ij}=B(x_{ij},r)\Big)_j$ is a packing of $G$ and
$\nu\big(B_{ij}\big)\leq r^\alpha$. Then we conclude from this that
$$
N_{\nu,r}^{1}(G)(2r)^0\leq \sum_{i,j} \nu(B_{i,j})\leq
2^{-\alpha}\xi_n M_r(G) (2r)^\alpha.
$$
Which implies that
$$
{\mathsf L}_{\nu}^{1,0}(G)\leq2^{-\alpha}\xi_n\overline{{\mathsf
P}}^{\alpha}(G)\leq2^{-\alpha} \xi_n\overline{\mathsf P}^{\alpha}(F)
$$
and
$$
\overline{\mathsf H}_{\nu}^{1,0}(F)\leq2^{-\alpha}
\xi_n\overline{\mathsf P}^{\alpha}(F).
$$
Assume that $F=\bigcup_iF_i$, then
$$
0<{\mathsf H}_{\nu}^{1,0}(F)\leq\sum_i\overline{\mathsf
H}_{\nu}^{1,0}(F_i)\leq2^{-\alpha} \xi_n\sum_i\overline{\mathsf
P}^{\alpha}(F_i).
$$
Finally, we conclude that $${\mathsf P}^{\alpha}(E)\geq{\mathsf
P}^{\alpha}(F)>0\quad\text{ and}\quad \overline{\dim}_{MB}(E)\geq
\alpha.$$ This completes the proof of Theorem \ref{thess}.
\end{proof}

\subsection{Proof of Theorems \ref{Jcole} and
\ref{ah1}}\label{Proofs} We present some intermediate results, which
will be used in the proof of Theorems \ref{Jcole} and \ref{ah1}. Let
$ \mu$ be a compactly supported Borel probability measure and $\nu$
be a finite Borel measure on $\mathbb{R}^n$ with
${K}=:\supp\mu=\supp\nu$. For $\zeta=(\mu,\nu)$, $E \subseteq {K}$
and $q,t\in \mathbb{R}$, we define
 $$
\overline{ {\mathsf P}}_\zeta^{q,t}(E)= \limsup_{r\to0}
M_{\zeta,r}^q(E) (2r)^t,
$$
where
$$
M_{\zeta,r}^q(E)=\sup \left\{\displaystyle \sum_i
\mu(B(x_i,r))^q\nu(B(x_i,r))\;\Big|\;\Big(B(x_i,r)\Big)_i\;\text{is
a centered packing of}\; E \right\}.
$$
Next we define the measure ${\mathsf{P}}_\zeta^{q,t}$ by
$$
{\mathsf{P}}_\zeta^{q,t}(E)=\inf \left\{\displaystyle
\sum_i{\overline{\mathsf P}}_\zeta^{q,t}(E_i)\;\Big|\;
E\subseteq\bigcup_i E_i\; \right\}.
 $$
In a similar way we define
$$
{\mathsf L}_\zeta^{q,t}(E)= \liminf_{r\to0} N_{\zeta,r}^q(E) (2r)^t,
$$
where
$$
N_{\zeta,r}^q(E)=\inf \left\{\displaystyle \sum_i
\mu(B(x_i,r))^q\nu(B(x_i,r))\;\Big|\;\Big(B(x_i,r)\Big)_i\;\text{is
a centered covering of}\; E \right\}.
$$
Since ${\mathsf L}_\zeta^{q,t}$  is  not increasing and not
countably sub-additive, one needs a standard modification to get an
outer measure. Hence, we modify the definition as follows
 $$
\overline{\mathsf{H}}_\zeta^{q,t}(E)=\sup_{F\subseteq E}
{\mathsf{L}}_\zeta^{q,t}(F)
 $$
and
$$
\mathsf{H}_\zeta^{q,t}(E)=\inf \left\{\displaystyle
\sum_i\overline{\mathsf H}_\zeta^{q,t}(E_i)\;\Big|\;
E\subseteq\bigcup_i E_i\; \right\}.
$$

The functions ${{\mathsf H}}^{q,t}_{\zeta}$ and ${{\mathsf
P}}^{q,t}_{\zeta}$ are metric outer measures. In addition, those
measures assign, in the usual way, a multifractal dimension to each
subset $E$ of $\mathbb{R}^n$. They are respectively denoted by
$\mathsf{b}_{\zeta}^{q}(E)$, $\mathsf{B}_{\zeta}^{q}(E)$ and
$\mathsf{\Delta}_{\zeta}^{q}(E)$. More precisely, we have
 $$
\begin{array}{lllcr}
\mathsf{b}_{\zeta}^{q}(E) &=&\inf \Big\{ t\in \mathbb{R}\;\bigl|\;\;
{\mathsf H}^{q,t}_{\zeta}(E) =0\Big\}, \\
\\\mathsf{B}_{\zeta}^{q}(E) &=& \inf \Big\{ t\in \mathbb{R}\;\bigl|\;\;\;
{\mathsf
P}^{q,t}_{\zeta}(E) =0\Big\}, \\\\
\mathsf{\Delta}_{\zeta}^{q}(E) &=& \inf \Big\{ t\in
\mathbb{R}\;\bigl| \quad \overline{{\mathsf P}}^{q,t}_{\zeta}(E)
=0\Big\}.
\end{array}
 $$
Observing that
 $$
\mathsf{b}_{\zeta}^{q}(E) \leq \mathsf{B}_{\zeta}^{q}(E) \leq
\mathsf{\Delta}_{\zeta}^{q}(E).
 $$
Next, we define the multifractal dimension functions
$\mathsf{b}_{\zeta}$, $\mathsf{B}_{\zeta}$ and
$\mathsf{\Delta}_{\zeta}$\;: $\mathbb{R}\rightarrow
[-\infty,+\infty]$ by
$$
\begin{array}{lll}
\mathsf{b}_{\zeta}:\;q\rightarrow \mathsf{b}_{\zeta}^{q}(K), \quad
\mathsf{B}_{\zeta}:\;q\rightarrow
\mathsf{B}_{\zeta}^{q}(K)\quad\text{and}\quad
\mathsf{\Delta}_{\zeta}:\;q\rightarrow\mathsf{\Delta}_{\zeta}^{q}(K).
\end{array}
$$
\begin{proposition}\label{P1} Set $f(t)=\mathsf{B}_{\zeta}({t})$ and suppose that
$f(0)=0$ and $\pi(K)>0$. Then
$$
\pi\left(\Big[E_{\mu}\big(-f_+'(0),-f_-'(0)\big)\Big]^{^{C}}\right)=0,
$$
where $f_-'$ and $f_+'$ are the left and right hand sides
derivatives of the function $f$.
\end{proposition}
\begin{proof} We must now prove that
$$
\pi\left(\Big\{x\in K\;\Bigl|\;\;
\overline{\alpha}_{\mu}(x)>-f_-'(0)\Big\}\right)=0.
$$
The proof of the statement
$$
\pi\left(\Big\{x\in
K\;\Bigl|\;\;\underline{\alpha}_{\mu}(x)<-f_+'(0)\Big\}\right)=0
$$
is identical to the proof of the statement in the first part and is
therefore omitted.

\bigskip
Fix $\alpha> -f_-'(0)$, then there exist $\beta,t>0$ such that
$\alpha>\beta> -f_-'(0)$ and $f(-t)< \beta t$. It follows
immediately from this that $ {\mathsf P}^{\;-t,  \beta
t}_{\zeta}(K)=0.$ We can choose a countable partition $(E_j)_j$ of
$K$, i.e., $K=\cup_j E_j$ such that
$$
\sum_j\overline{\mathsf P}^{\;-t,  \beta t}_{\zeta}(E_j)\leq 1
\quad\text{and}\quad\overline{\mathsf P}^{\;-t,  \alpha
t}_{\zeta}(E_j)=0,\;\; \forall\;j.
$$
Now put
$$
F=\left\{x\in K\;\Biggl|\;\; \limsup_{r\to0} \frac{\log
\mu\big(B(x,r)\big)}{\log r}>\alpha\right\}.
$$
If $x \in F$, for all $\delta>0$, there exists $0<r_x<\delta$ such
that
$$
\mu\big(B(x,r)\big)\leq r_x^\alpha.
$$
Fix $F_j=E_j\cap F$ and $G_j\subset F_j$. For all $\delta>0$ and all
$j$, the family $\Big(B(x_{j}, r_{x_{j}})\Big)_{x_j\in G_j}$ is a
centered $\delta$-covering of $\overline{G}_j$. Then, we can choose
a finite subset $J$ of $\N$ such that the family
$\Big(B(x_{j_i},r_{j_i})\Big)_{i\in J}$ is a centered
$\delta$-covering of $G_j$. Take $r_j = \max\{ r_{j_i},\;{i\in
J}\}$, then $\Big(B(x_{j_i} , r_j)\Big)_{ i\in J}$  is a centered
covering of $G_j.$ From Besicovitch's covering theorem, we can
construct $\xi_n$ finite sub-families
$\Big(B(x_{j_{1k}},r_{j})\Big)_k$,....,$\Big(B(x_{j_{\xi_n
k}},r_{j})\Big)_k$ such that each  $$
G_j\subseteq\displaystyle\bigcup_{i=1}^{\xi_n}\bigcup_kB(x_{j_{ik}},r_{j}),$$
$\Big(B_{j_{ik}}=B(x_{j_{ik}},r_{j})\Big)_k$ is a packing of $G_j$
and $\mu\big(B_{j_{ik}}\big)\leq r_j^\alpha$. We have
\begin{eqnarray*}
N_{\nu,r_j}^{1}(G_j)(2r_j)^0\leq \sum_{i,k}\nu
(B_{j_{ik}})&=&\sum_{i,k}\mu\big(B_{j_{ik}}\big)^{-t}\mu\big(B_{j_{ik}}\big)^t\nu
\big(B_{j_{ik}}\big)\\
&\leq&2^{-\alpha t}\sum_{i,k}\mu\big(B_{j_{ik}}\big)^{-t}\nu
\big(B_{j_{ik}}\big)(2r_j)^{\alpha t}.
\end{eqnarray*}
We deduce that
$$
{\mathsf L}_{\nu}^{1,0}(G_j)\leq2^{-\alpha t}\xi_n\overline{\mathsf
P}^{\;-t, \alpha t}_{\zeta}(E_j).
$$
It follows that
$$
\overline{\mathsf H}^{1, 0}_{\nu}(F_j)\leq2^{-\alpha
t}\xi_n\overline{\mathsf P}^{\;-t, \alpha t}_{\zeta}(E_j)=0.
$$
Finally, we immediately conclude that $$\pi(F)={\mathsf H}^{1,
0}_{\nu}(F)\leq\sum_j\overline{\mathsf H}^{1, 0}_{\nu}(F_j)=0.$$
This completes the proof of Proposition \ref{P1}.
\end{proof}
\begin{proposition}\label{P2} With the same notations and hypotheses as in the previous
proposition, we have
$$
\underline{\dim}_{MB}\Big(E_{\mu}\big(-f_+'(0),-f_-'(0)\big)\Big)\geq
\inf\left\{ \underline{\alpha}_{\nu}(x)\;\Biggl|\;\;x\in
E_{\mu}\Big(-f_+'(0),-f_-'(0)\Big) \right\}
$$
and
$$
\overline{\dim}_{MB}\Big(E_{\mu}\big(-f_+'(0),-f_-'(0)\big)\Big)\geq
\inf\left\{ \overline{\alpha}_{\nu}(x)\;\Biggl|\;\;x\in
E_{\mu}\Big(-f_+'(0),-f_-'(0)\Big) \right\}.
$$
\end{proposition}
\begin{proof}
It follows immediately from Theorem \ref{thess} and Proposition
\ref{P1}.
\end{proof}

\bigskip\bigskip\bigskip
{\it Proof of Theorems \ref{Jcole} and \ref{ah1}.} We can now prove
Theorems \ref{Jcole} and \ref{ah1}. For $q\leq0$, take
$$ \nu\big(B(x,r)\big)= \mu\big(B(x,r)\big)^q (2r)^{{\mathsf
B}_{\mu}(q)}.$$ By a simple calculation, we get
$$\mathsf{B}_{\zeta}({t})=\mathsf{B}_{\mu}(q+t)-\mathsf{B}_{\mu}(q)$$ and if $x \in
\underline{E}_{\mu}(\alpha)$, we have
$$
\underline{\alpha}_{\nu}(x)=q
\underline{\alpha}_{\mu}(x)+\mathsf{B}_{\mu}(q)\leq
q\alpha+\mathsf{B}_{\mu}(q).
$$
Theorem \ref{thinf} clearly implies that
$$
\overline{\dim}_{MB}\Big(\underline{E}_{\mu}(\alpha)\Big)\leq
\inf_{q\leq0} q\alpha+\mathsf{B}_{\mu}(q).
$$
The proof of the statement
$$
\overline{\dim}_{MB}\Big(\overline{E}_{\mu}(\alpha)\Big)\leq
\inf_{q\geq0} q\alpha+\mathsf{B}_{\mu}(q)
$$
is very similar to the proof of the first statement and is therefore
omitted.

\bigskip\bigskip
We now turn to lower bound theorems. If moreover we suppose that
$\mathsf{H}_{\mu}^{q,\mathsf{B}_{\mu}(q)}\big( K\big)>0,$ then
clearly we have $\pi(K)>0$. By taking Proposition \ref{P1} into
consideration, we get
$$
\pi\Big(E_{\mu}\big(-\mathsf{B}_{\mu_+}'(q),-\mathsf{B}_{\mu_-}'(q)\big)\Big)>0.
$$
It follows immediately from Proposition \ref{P2} that
$$
\underline{\dim}_{MB}\Big(E_{\mu}\big(-\mathsf{B}_{\mu_+}'(q),-\mathsf{B}_{\mu_-}'(q)\big)\Big)\geq
\left\{
  \begin{array}{ll}
  -q\mathsf{B}_{\mu_+}'(q)+\mathsf{B}_{\mu}(q), & \hbox{if } q\geq0 \\\\
-q\mathsf{B}_{\mu_-}'(q)+\mathsf{B}_{\mu}(q), & \hbox{if } q\leq0.
  \end{array}
\right.
$$
Finally, it follows from this and since $\mathsf{B}_{\mu}$ is
differentiable at $q$ that
 $$
\underline{\dim}_{MB}
\Big(E_{\mu}(-\mathsf{B}'_{\mu}(q))\Big)=\overline{\dim}_{MB}
\Big(E_{\mu}(-\mathsf{B}'_{\mu}(q))\Big)=
\mathsf{B}^*_{\mu}\big(-\mathsf{B}'_{\mu}(q)\big)=\mathsf{b}^*_{\mu}\big(-\mathsf{B}'_{\mu}(q)\big),
 $$
which yields the desired result.

\begin{theorem}
The previous results hold if we replace the multifractal function
$\mathsf{B}_{(.)}$ by the function $\mathsf{\Delta}_{(.)}$.
\end{theorem}

\subsection{A result for which the Hewitt-Stromberg dimensions
differ}\label{differ} We now focus on the case when the upper and
the lower Hewitt-Stromberg dimension functions do not necessarily
coincide, i.e., $\mathsf{B}_{\mu}(q)\neq \mathsf{b}_{\mu}(q)$ for
$q\neq 1$. Consider the sets, for all $\alpha,\beta\geq0$
\begin{eqnarray*}
E(\alpha,\beta)=\Big\{x\in
K\;\Bigl|\;\;\underline{\alpha}_{\mu}(x)\leq \alpha
\;\;\text{and}\;\;\beta\leq\overline{\alpha}_{\mu}(x)\Big\}\quad\text{and}\quad
E(\alpha)=E(\alpha,\alpha).
\end{eqnarray*}
For a real function $\varphi$, we set
$$
\varphi^-(q)=\limsup_{t\to0}\frac{\varphi(q-t)-\varphi(q)}{-t}
$$
and
$$
\varphi^+(q)=\limsup_{t\to0}\frac{\varphi(q+t)-\varphi(q)}{t}.
$$
\begin{theorem}\label{th4}
For $q\in \mathbb{R}$, assume that
$\mathsf{H}_{\mu}^{q,\mathsf{b}_{\mu}(q)}\big( K\big)>0,$ then we
have
$$
\overline{\dim}_{MB}\Big(E\big(-\mathsf{b}_{\mu}^-(q),\;
-\mathsf{b}_{\mu}^+(q) \big)\Big)\geq \left\{
  \begin{array}{ll}
  -q\mathsf{b}_{\mu}^+(q)+\mathsf{b}_{\mu}(q), & \hbox{if } q\geq0, \\\\
-q\mathsf{b}_{\mu}^-(q)+\mathsf{b}_{\mu}(q), & \hbox{if } q\leq0.
  \end{array}
\right.
$$
In addition, if $\mathsf{b}_{\mu}$ is differentiable at $q$, one has
 $$
\overline{\dim}_{MB}\Big(E\big(-\mathsf{b}_{\mu}'(q)\big)\Big)\geq
-q\mathsf{b}_{\mu}'(q)+\mathsf{b}_{\mu}(q).
 $$
\end{theorem}
Theorem \ref{th4} is a consequence from Theorem \ref{thess} and the
following proposition.
\begin{proposition}\label{P3}
Let $\varphi(t)=\mathsf{b}_{\zeta}({t})$ and suppose that
$\varphi(0)=0$ and $\pi(K)>0$. Then
$$
\pi\left(\Big\{x\in K\;\Bigl|\;\;
\underline{\alpha}_{\mu}(x)>-\varphi^-(0)\Big\}\right)=0
$$
and
$$
\pi\left(\Big\{x\in K\;\Bigl|\;\;
\overline{\alpha}_{\mu}(x)<-\varphi^+(0)\Big\}\right)=0.
$$

\end{proposition}
\begin{proof}
We will prove the first assertion.  The proof of the second
statement is identical to the proof of the statement in the first
part and is therefore omitted.

\bigskip
Let $\alpha>-\varphi^-(0)=\liminf_{t\to0}\frac{\varphi(-t)}{t}$,
then there exists $t>0$ such that $\alpha t>\varphi(-t)$. It is
clear that $ {\mathsf H}^{\;-t,  \alpha t}_{\zeta}(K)=0.$ Consider
the following set
$$
E=\left\{x\in K\;\Biggl|\;\; \liminf_{r\to0} \frac{\log
\mu\big(B(x,r)\big)}{\log r}>\alpha\right\}.
$$
For all $x\in E$, we can therefore choose $\delta>0$ such that, for
all $0<r<\delta$, we have
$$
\mu\big(B(x,r)\big)<r^\alpha.
$$
Now, let $(E_j)_j$ be a countable partition  of $E$. We put forward
the set
$$
E_{m_j}=\left\{x\in E_j\;\Biggl|\;\; \text{for all}\; r<
\frac1m,\;\;\mu\big(B(x,r)\big)<r^\alpha\right\}.
$$
Fix $F\subset E_{m_j}$ and $0<r<\min(\delta,\frac1m)$.  Let
$\Big(B_i=B(x_i,r)\Big)_i$ be a centered covering of $F$. Then
\begin{eqnarray*}
N_{\nu,r}^{1}(F)(2r)^0\leq \sum_i\nu
(B_{i})&=&\sum_i\mu\big(B_i\big)^{-t}\mu\big(B_i\big)^t\nu
\big(B_i\big)\\
&\leq&2^{-\alpha t}\sum_i\mu\big(B_i\big)^{-t}\nu
\big(B_i\big)(2r)^{\alpha t}.
\end{eqnarray*}
It follows immediately from this that
$$
N_{\nu,r}^{1}(F)(2r)^0\leq2^{-\alpha t}
N_{\zeta,r}^{-t}(F)(2r)^{\alpha t}.
$$
Letting $r$ tend to $0$, gives
$$
{\mathsf L}^{1, 0}_{\nu}(F)\leq2^{-\alpha t}{\mathsf L}^{\;-t,
\alpha t}_{\zeta}(F)\leq2^{-\alpha t}\overline{\mathsf H}^{\;-t,
\alpha t}_{\zeta}(E_{m_j}).
$$
We therefore conclude that
$$\overline{\mathsf H}^{1,
0}_{\nu}(E_{m_j})\leq2^{-\alpha t}\overline{\mathsf H}^{\;-t, \alpha
t}_{\zeta}(E_{m_j})$$ and
$$
{\mathsf H}^{1, 0}_{\nu}(E)\leq\sum_j\sum_{m}\overline{\mathsf
H}^{1, 0}_{\nu}(E_{m_j})\leq2^{-\alpha
t}\sum_j\sum_{m}\overline{\mathsf H}^{\;-t, \alpha
t}_{\zeta}(E_{m_j}).
$$
Finally, we get
$$
{\mathsf H}^{1, 0}_{\nu}(E)\leq 2^{-\alpha} {\mathsf H}^{\;-t,
\alpha t}_{\zeta}(E)\leq 2^{-\alpha} {\mathsf H}^{\;-t, \alpha
t}_{\zeta}(K)=0.
$$
This completes the proof of Proposition \ref{P3}.
\end{proof}

\begin{remark}
Let $(\mathbb{X},\rho)$ be a separable metric space, $\mathscr{B}$
stand for the set of balls of $\mathbb{X}$, and $\mathscr{F}$ for
the set of maps from $\mathscr{B}$ to $[0, +\infty)$. The set of
$\mu\in\mathscr{F}$ such that $\mu(B) = 0$ implies $\mu(B') = 0$ for
all $B'\subseteq B$ will be denoted by $\mathscr{F}^*$. For such a
$\mu$, one defines its support $\supp\mu$ to be the complement of
the set $\bigcup\{B\in \mathscr{B}\;\bigl|\;\;\mu(B)=0\}$.  If
moreover, we assume $(\mathbb{X},\rho)$ having the Besicovitch
property then all the above results hold for any function $\mu$ in
$\mathscr{F}^*$.
\end{remark}

\begin{remark} Notice that our formalism is adapted with the
vectorial multifractal formalism introduced by Peyri\`{e}re in
\cite{Pee} and in particular, the mixed multifractal formalism
introduced in \cite{Menceur2}. A valuation on the metric space
$(\mathbb{X}, \rho)$ is a real function $\xi$ defined on the set of
balls of $\mathbb{X}$, subject to the condition that it goes to
$+\infty$ as the radius goes to 0. In \cite{Pee}, Peyri\`{e}re
defined a more general multifractal formalism by considering
quantities of type
$$
\sum_{i} e^{-\Big(\big<q, \chi\left(x_{i}, r_{i}\right)\big>+t
\xi\left(x_{i}, r_{i}\right)\Big)}
$$
where $\chi: \mathbb{X} \times [0, +\infty)\mapsto
\mathbb{E}^{\prime}$ is a function, $\mathbb{E}^{\prime}$ is the
dual of a separable real Banach space $\mathbb{E},\big<.,.\big>$ is
the duality bracket between $\mathbb{E}$ and $\mathbb{E}^{\prime}$
and $(\mathbb{X}, \rho)$ is a metric space where the Besicovitch
covering theorem holds. Here we can introduce  a multifractal
formalism based on the vectorial Hewitt-Stromberg measures and that
this formalism is completely parallel to Peyri\`{e}re's multifractal
formalism which based on the  vectorial Hausdorff and packing
measures: While the vectorial measures are defined using coverings
and packings by families of sets with diameters less than a given
positive number $r$, say, the Hewitt-Stromberg measures are defined
using packings of balls with a fixed diameter $r$. More precisely,
by considering the quantities of type
$$
\sum_{i} e^{-\Big(\big<q, \chi\left(x_{i}, r\right)\big>+t
\xi\left(x_{i}, r\right)\Big)},
$$
then the results of \cite{NB2} and all above main theorems hold for
this vectorial (mixed) multifractal formalism for some prescribed
functions $\chi$ and $\xi$.
\end{remark}

\section{Some examples}\label{sec4}

This section discusses more motivations and examples related to
these concepts. We present an intermediate result, which will be
used in the proof of our results.
\begin{lemma}\label{Lemma} We have
\begin{eqnarray*}
\mathsf{\Delta}_{\zeta}^{q}(E)&=&\limsup_{r\to 0}\frac{\log
M_{\zeta,r}^q(E)}{-\log r}\\&=&\limsup_{r\to 0}\frac{1}{-\log
r}\log\left(\sup \left\{\displaystyle \sum_i
\mu(B(x_i,r))^q\nu(B(x_i,r))\;\Big|\;\Big(B(x_i,r)\Big)_i\;\text{is
a packing of}\; E \right\}\right).
\end{eqnarray*}
\end{lemma}
\begin{proof}
 Suppose
that $$\limsup_{r\to 0}\frac{\log M_{\zeta,r}^q(E)}{-\log r} >
\mathsf{\Delta}_{\zeta}^{q}(E) +\epsilon\quad\text{ for
some}\quad\epsilon
>0.$$ Then we can find $\delta >0$ such that for any $r\le \delta$,
$$
M_{\zeta, r}^q (E)~ (2r)^{{  \mathsf{\Delta}_{\zeta}^{q}(E)
+\epsilon}}
>2^{{  \mathsf{\Delta}_{\zeta}^{q}(E)
+\epsilon}}$$ and then $$\overline{ {\mathsf P}}_\zeta^{q,
\mathsf{\Delta}_{\zeta}^{q}(E) +\epsilon}(E) \ge 2^{{
\mathsf{\Delta}_{\zeta}^{q}(E) +\epsilon}}>0
$$
which is a contradiction. We therefore infer $$ \limsup_{r\to
0}\frac{\log M_{\zeta,r}^q(E)}{-\log r}\le
\mathsf{\Delta}_{\zeta}^{q}(E) +\epsilon\;\;\text{ for any}\;\;
\epsilon
>0.$$ The proof of the following statement  $$
\limsup_{r\to 0}\frac{\log M_{\zeta,r}^q(E)}{-\log r}  \ge
\mathsf{\Delta}_{\zeta}^{q}(E) - \epsilon\;\;\text{ for
any}\;\;\epsilon
>0$$
is identical to the proof of the above statement and is therefore
omitted.
\end{proof}

\noindent{\bf Moran sets:} Let us recall the class of homogeneous
Moran sets (see for example \cite{W1, W3, W4}). We denote by
$\{n_k\}_{k\ge1}$ a sequence of positive integers and
$\{c_k\}_{k\ge1}$ a sequence of positive numbers satisfying
\begin{center} $n_k\ge2,\quad 0<c_k<1,\quad n_kc_k\le1\quad$ for
$k\ge1$.\end{center} Let $D_0=\emptyset$, and for any $k\ge1$, set
$$
D_{m, k}=\Big\{\left(i_{m}, i_{m+1}, \ldots, i_{k}\right) ; \quad 1
\leqslant i_{j} \leqslant n_{j}, \quad m \leqslant j \leqslant
k\Big\}\quad\text{and}\quad D_k=D_{1,k}.$$ Define
$D=\displaystyle\bigcup_{k\geqslant1}D_{k}.$ If
$$\sigma=\left(\sigma_{1}, \ldots, \sigma_{k}\right) \in D_{k}$$
and $$\tau=\left(\tau_{1}, \ldots,\tau_{m}\right) \in D_{k+1, m},$$
then we denote $$\sigma * \tau=\left(\sigma_{1}, \ldots, \sigma_{k},
\tau_{1}, \ldots, \tau_{m}\right).$$
\begin{definition}
Let $J$ be a closed interval such that $|J|=1$. We say the
collection $\mathscr{F}=\left\{J_{\sigma}, \sigma \in D\right\}$ of
closed subsets of $J$ fulfills the Moran structure if it satisfies
the following conditions:
\\ \text{(a)} $J_{\emptyset}=J.$
\\ \text{(b)} For all $k\ge0$ and $\sigma \in D_{k}, J_{\sigma * 1}, J_{\sigma
* 2}, \ldots, J_{\sigma{*} n_{k+1}}$ are subintervals of $J_\sigma$, and
satisfy that $$J_{\sigma * i}^{\circ} \cap J_{\sigma
*j}^{\circ}=\emptyset \;\;\;\text{for all}\;\;i \neq j,$$ where $A^{\circ}$ denotes the interior
of $A.$
\\ \text{(c)} For any
$k\ge1,\sigma \in D_{k-1}, \quad
c_{k}=\displaystyle\frac{\left|J_{\sigma{*} j}\right|}
{\left|J_{\sigma}\right|}, \quad 1 \leqslant j \leqslant n_{k}$
where $|A|$ denotes the diameter of $A.$
\end{definition}
Let $\mathscr{F}$ be a collection of closed subintervals of $J$
having homogeneous Moran structure. The set
$$E(\mathscr{F})=\displaystyle\cap_{k \geqslant 1}\cup_{\sigma \in
D_{k}} J_{\sigma}$$ is called an homogeneous Moran set determined by
$\mathscr{F}$. It is convenient to denote
$M\left(J,\left\{n_{k}\right\},\left\{c_{k}\right\}\right)$ for the
collection of homogeneous Moran sets determined by $J,$
$\left\{n_{k}\right\}$ and $\left\{c_{k}\right\}$.

\begin{remark}
If $\displaystyle\lim _{k \rightarrow +\infty} \sup _{\sigma \in
D_{k}}\left|J_{\sigma}\right|>0$, then E contains interior points.
Thus the measure and dimension properties will be trivial. We assume
therefore $\displaystyle\lim _{k\rightarrow +\infty} \sup _{\sigma
\in D_{k}}\left|J_{\sigma}\right|=0.$
\end{remark}
Now, we consider a class of homogeneous Moran sets $E$ witch satisfy
a special property called the strong separation condition
(\texttt{SSC}), i.e., take $J_\sigma\in {\mathscr F}$. Let
$J_{\sigma
*1}, J_{\sigma
*2}, \ldots, J_{\sigma *n_{k+1}}$ be the $n_{k+1}$ basic intervals of order
$k + 1$ contained in $J_\sigma$ arranged from the left to the right,
then we assume that for all $1\leq i\leq n_{k+1}-1$, $ \text{dist}
(J_{\sigma *i}, J_{\sigma *(i+1)}) \geq \Delta_k |J_\sigma|,$ where
$(\Delta_k)_k$ is a sequence of positive real numbers, such that $$
0<\Delta=\displaystyle\inf_k \Delta_k<1.$$
\subsection{Example 1} In this example, we discuss our multifractal structures of a
class of regularity Moran fractals associated with the sequences of
letters such that the frequency exists. Let $A=\{a, b\}$ be a
two-letter alphabet, and  $A^{*}$ the free monoid generated by $A$.
Let $F$ be the homomorphism on $A^{*},$ defined by $F(a)=ab$ and
$F(b)=a$. It is easy to see that $F^{n}(a)=F^{n-1}(a) F^{n-2}(a)$.
We denote by $\left|F^{n}(a)\right|$ the length of the word
$F^{n}(a)$, thus
$$F^{n}(a)=s_{1} s_{2} \cdots
s_{\left|F^{n}(a)\right|},\quad s_{i} \in A.$$ Therefore, as $n
\rightarrow \infty$, we get the infinite sequence
$$\omega=\lim _{n \rightarrow +\infty} F^{n}(a)=s_{1} s_{2} s_{3}
\cdots s_{n} \cdots \in\{a, b\}^{\mathbb{N}}$$ which is called the
Fibonacci sequence. For any $n\geqslant 1,$ write
$\omega_{n}=\left.\omega\right|_{n}=s_{1} s_{2} \cdots s_{n}$. We
denote by $\left|\omega_{n}\right|_{a}$ the number of the occurrence
of the letter $a \text { in } \omega_{n}, \text { and
}\left|\omega_{n}\right|_{b}$ the number of occurrence of $b.$ Then
$\left|\omega_{n}\right|_{a}+\left|\omega_{n}\right|_{b}=n$. It
follows from \cite{W1} that $\displaystyle\lim _{n \rightarrow+
\infty} \frac{\left|\omega_{n}\right|_{a}}{n}=~\eta,$  where
$\eta^{2}+\eta=1$.

 \bigskip
Let $0<r_{a}<\frac{1}{2}, 0<r_{b}<\frac{1}{3}, r_{a}, r_{b} \in
\mathbb{R}$. In the above Moran construction, let
$$|J|=1, \quad n_{k}=\left\{\begin{array}{ll}{2,} & {\text { if } s_{k}=a}
 \\ \\{3,} & {\text { if }
s_{k}=b}\end{array}\right.$$ and
$$c_{k}=\left\{\begin{array}{ll}{r_{a},} & {\text { if } s_{k}=a}
 \\ \\{r_{b},} & {\text { if } s_{k}=b}\end{array}, \quad 1 \leqslant j \leqslant
n_{k}.\right.$$ Then we construct the homogeneous Moran set relating
to the Fibonacci sequence and denote it by
$E:=E(\omega)=\left(J,\left\{n_{k}\right\},\left\{c_{k}\right\}\right)$.
By the construction of $E,$ we have
$$\left|J_{\sigma}\right|=r_{a}^{\left|\omega_{k}\right|_{a}}
r_{b}^{\left|\omega_{k}\right|_{b}},\quad\forall \sigma\in D_{k}.$$

 \bigskip
Let $P_{a}=\left(P_{a_{1}}, P_{a_{2}}\right), P_{b}=\left(P_{b_{1}},
P_{b_{2}}, P_{b_{3}}\right)$ be probability vectors, i.e.,
$$P_{a_{i}}>0,\;\; P_{b_{i}}>0,\quad\text{and}\quad\displaystyle\sum_{i=1}^{2}
P_{a_{i}}=1, \;\;\sum_{i=1}^{3} P_{b_{i}}=1.$$ For any $k \geqslant
1$ and any $\sigma \in D_{k},$ we know $\sigma=\sigma_{1} \sigma_{2}
\cdots \sigma_{k}$ where
$$\sigma_{k}\in\left\{\begin{array}{ll}{\{1,2\},} &
{\text { if } s_{k}=a} \\\\ {\{1,2,3\},} & {\text { if }
s_{k}=b.}\end{array}\right.$$ For $\sigma=\sigma_{1}\sigma_{2}
\cdots\sigma_{k},$ we define $\sigma(a)$ as follows: let
$\omega_{k}=s_{1}s_{2}\cdots{s_{k}}$ and
${e_{1}}<e_{2}<\cdots<e_{\left|\omega_{k}\right|_{a}}$ be the
occurrences of the letter a in $\omega_{k},$ then
$\sigma(a)=\sigma_{e_{1}}\sigma_{e_{2}}\cdots
\sigma_{e_{\left|\omega_{k}\right|_{a}}}$. Similarly, let
$\delta_{1}<\delta_{2}<\cdots<\delta_{\left|\omega_{k}\right|_{b}}$
be the occurrences of the letter b in $\omega_{k},$ then
$\sigma(b)=\sigma_{\delta_{1}}\sigma_{\delta_{2}}\cdots
\sigma_{\delta_{|\omega_{k}|_{b}}}$.
\\Let
$${P_{\sigma(a)}=P_{\sigma_{e_{1}}}
P_{\sigma_{e_{2}}} \cdots
P_{\sigma_{e\left|\omega_{k}\right|_{a}}}}$$ and
$${P_{\sigma(b)}=P_{\sigma_{\delta_{1}}} P_{\sigma_{\delta_{2}}}
\cdots P_{\sigma_{\delta\left|\omega_{k}\right|_{b}}}}.$$ Obviously
$$\displaystyle\sum_{\sigma\in D_{k}}P_{\sigma(a)}
P_{\sigma(b)}=1.$$ Let $\mu$ be a mass distribution on E, such that
for any $\sigma \in D_{k},$
$$\mu\left(J_{\sigma}\right)=P_{\sigma(a)} P_{\sigma(b)}.$$

Now we define an auxiliary function $\beta(q)$ as follows: For each
$q\in\mathbb{R}$ and $k\ge1$, there is a unique number
$\beta_{k}(q)$ such that
\begin{equation*}
\sum_{\sigma \in D_{k}}\Big(P_{\sigma(a)}
P_{\sigma(b)}\Big)^{q}\left|J_{\sigma}\right|^{\beta_{k}(q)}=1.
\end{equation*}
By a simple calculation, we get
$$\beta_{k}(q)=\frac{-\log \left(\displaystyle\sum_{i=1}^{2} P_{a_{i}}^{q}\right)
-\frac{k-\left|\omega_{k}\right|_{a}}{\left|\omega_{k}\right|_{a}}
\log \left(\displaystyle\sum_{i=1}^{3} P_{b_{i}}^{q}\right)}{\log
r_{a}+\frac{k-\left|\omega_{k}\right|_{a}}{\left|\omega_{k}\right|_{a}}
\log r_{b}}.$$ Clearly, for any  $k \geqslant 1$ we have
$\beta_{k}(1)=0.$
 Thus $\beta_{k}^{\prime}(q)<0$ for all $q$ and
$\beta_{k}(q)$ is a strictly decreasing function. Our auxiliary
function is
$$
\beta(q)=\lim _{k \rightarrow +\infty} \beta_{k}(q)=\frac{-\log
\left(\displaystyle\sum_{i=1}^{2} P_{a_{i}}^{q}\right)-\eta \log
\left(\displaystyle\sum_{j=1}^{3} P_{b_{j}}^{q}\right)}{\log
r_{a}+\eta \log r_{b}},
$$
$\text { where } \eta^{2}+\eta=1$. The function $\beta(q)$ is
strictly decreasing, smooth, $\lim _{q \rightarrow \mp \infty}
\beta(q)=\pm \infty$ and $\beta(1)=0$.

\bigskip
Then we have the following result,
\begin{theorem} Suppose that $E$ is a homogeneous Moran set satisfying
(\texttt{SSC}) and $\mu$ is the Moran measure on $E$. Then,
$$
\underline{\dim}_{MB}\big(E_\mu({-\beta'(q)})\big)=\overline{\dim}_{MB}\big(E_\mu({-\beta'(q)})\big)=\beta^*(-\beta'(q)).
$$
\end{theorem}
\begin{proof}  Given
$q,t\in\mathbb{R},$ and let $\nu$ be a probability measure on $E$
such that for any $k\ge1$ and $\sigma_0\in D_k$,
$$\nu\left(J_{\sigma_{0}}\right)=\frac{\mu\left(J_{\sigma_{0}}\right)^{q}\left|J_{\sigma_{0}}\right|^{t}}
{\displaystyle\sum_{\sigma\in{D}_{k}}\mu\left(J_{\sigma}\right)^{q}\left|J_{\sigma}\right|^{t}}.$$
It follows from Proposition \ref{Prop1} that
$$
{\mathsf B}_{\nu}^{1}\big(E_\mu({-\beta'(q)})\big)\leq {\mathsf
B}_{\nu}({1})=0.
$$
It result from Theorem \ref{thinf} that
\begin{eqnarray}\label{Moran1}
\overline{\dim}_{MB}\big(E_\mu({-\beta'(q)})\big)\leq
-q\beta'(q)+\beta(q).
\end{eqnarray}
Now, by using Lemma \ref{Lemma}, we define the function $f$ by
$$
f(t)=\limsup_{r\to 0}\frac{1}{-\log r}\log\left(\sup
\left\{\displaystyle \sum_i
\mu(B(x_i,r))^t\nu(B(x_i,r))\;\Big|\;\Big(B(x_i,r)\Big)_i\;\text{is
a packing of}\; \supp\mu \right\}\right).
$$
Then $f(t)=\lim_{k\to+\infty} f_k(t)$, where $f_k(t)$ is a unique
number such that
\begin{equation*}
\sum_{\sigma \in D_{k}}\left(\mu(J_{\sigma})\right)^{t}
\nu(J_{\sigma})\left|J_{\sigma}\right|^{f_{k}(t)}=1.
\end{equation*}
Which implies that
\begin{equation*}
\sum_{\sigma \in
D_{k}}\left(\mu(J_{\sigma})\right)^{q+t}\left|J_{\sigma}\right|^{f_{k}(t)+\beta_k(q)}=1.
\end{equation*}
A straightforward calculation shows that
$$f_{k}(t)=\frac{-\log \left(\displaystyle\sum_{i=1}^{2} P_{a_{i}}^{q+t}\right)
-\frac{k-\left|\omega_{k}\right|_{a}}{\left|\omega_{k}\right|_{a}}
\log \left(\displaystyle\sum_{i=1}^{3} P_{b_{i}}^{q+t}\right)}{\log
r_{a}+\frac{k-\left|\omega_{k}\right|_{a}}{\left|\omega_{k}\right|_{a}}
\log r_{b}}- \beta_k(q)$$ and
$$
f(t)=\lim_{k\to+\infty} f_k(t)=\beta(q+t)- \beta(q).
$$
It is clear that $f(0) = 0$, $f'(0)$ exists and equal to
$\beta'(q)$. We can see from the construction of measure $\nu$ that
$$0<\displaystyle\lim_{k\to +\infty}\sum_{\sigma\in{D}_{k}}
\mu\left(J_{\sigma}\right)^{q}\left|J_{\sigma}\right|^{{\beta}(q)}<+\infty$$
and then
$$
\nu(\supp\mu)>0 \;\Rightarrow \; \pi(\supp\mu)>0.
$$
We therefore conclude from Proposition \ref{P2} that
\begin{eqnarray}\label{Moran2}
\underline{\dim}_{MB}\big(E_\mu({-\beta'(q)})\big)\geq
-q\beta'(q)+\beta(q).
\end{eqnarray}
Thus, the result is a consequence from \eqref{Moran1} and
\eqref{Moran2}.
\end{proof}
\subsection{Example 2} In the following, our results are for a particular type of
fractal called a non-regularity Moran fractal associated with the
sequences of which the frequency of the letter does not exists. We
first define a Moran measure on the homogeneous Moran sets $E$. Let
$\big\{p_{i,j}\big\}_{j=1}^{n_i}$ be the probability vectors (i.e.,
$p_{i,j}>0$ and $\sum_{j=1}^{n_i}p_{i,j}=1$ for all $i=1,2,3,...$)
and suppose that $p_0=\inf \{p_{i,j}\}>0$.  Let $\mu$ be a mass
distribution on $E$, such that for any $J_\sigma$ ($\sigma\in D_k$)
$$
\mu(J_\sigma)=p_{1,\sigma_1}p_{2,\sigma_2}\cdots p_{k,\sigma_k}.
$$
Now, we define an auxiliary function $\beta_k(q)$ as follows: for
all $k \geq 1$ and $q\in \mathbb{R}$, there is a unique number
$\beta_k(q)$ satisfying
$$
\sum_{\sigma\in D_k} p_\sigma^q |J_\sigma|^{\beta_k(q)}=1.
$$
Set
$$
\underline{\beta}(q)=\liminf_{k\to+\infty}\beta_k(q)\quad\text{and}\quad\overline{\beta}(q)=\limsup_{k\to+\infty}\beta_k(q).
$$
Let $(t_k)_k$  be a sequence of integers such that
$$
t_1=1,\;\; t_2=3\;\;\text{and}\;\; t_{k+1}=2t_k,\;\; \forall\;
k\geq3.
$$
Define the family of parameters $n_i$, $c_i$ and $p_{i,j}$ as
follows:
$$
n_1 = 2,\;\;n_i=\left\{
  \begin{array}{ll}
  3,\;\text{if}\;\; t_{2k-1}\leq i<t_{2k}, &\\
\\
  2,\;\text{if}\;\; t_{2k}\leq i<t_{2k+1}. &
  \end{array}
\right.
$$
For $0<r_a<\frac12$ and  $0<r_b<\frac13$, let
$$
c_1 = r_a,\;\;c_i=\left\{
  \begin{array}{ll}
  r_b,\;\text{if}\;\; t_{2k-1}\leq i<t_{2k}, &\\
\\
  r_a,\;\text{if}\;\; t_{2k}\leq i<t_{2k+1}. &
  \end{array}
\right.
$$
Let $(p_{a,j})_{j=1}^2$ and  $(p_{b,j})_{j=1}^3$ be two  probability
vectors. Define
$$
p_{1,j}=p_{a,j},\;\text{for all}\; 1\leq j\leq 2$$ and
$$
p_{i,j}=\left\{
  \begin{array}{ll}
  p_{b,j},\;\text{if}\;\; t_{2k-1}\leq i<t_{2k},\;\;1\leq j\leq
3, &\\
\\
  p_{a,j},\;\text{if}\;\; t_{2k}\leq i<t_{2k+1},\;\;1\leq j\leq
2. &
  \end{array}
\right.
$$
Then, we have
$$
\beta_k(q)=\frac{\log \sum_{\sigma\in D_k}\mu(J_\sigma)^q}{-\log
(c_1\cdots c_k)}.
$$
Finally, if $N_k$ is the number of integers $i\leq k$ such that
$p_{i,j} = p_{a,j}$, then $$\liminf_{k\to+\infty}
\frac{N_k}{k}=\frac{1}{3},$$
$$\limsup_{k\to+\infty} \frac{N_k}{k}=\frac{2}{3}$$ and
$$
\beta_{k}(q)=-\frac{\frac{N_k}k\log
\left(\displaystyle\sum_{j=1}^{2} p_{a,j}^{q}\right)
+\left(1-\frac{N_k}k\right)\log \left(\displaystyle\sum_{j=1}^{3}
p_{b,j}^{q}\right)}{\frac{N_k}k\log
r_{a}+\left(1-\frac{N_k}k\right)\log r_{b}}.
$$
We can then conclude that
$$
\underline{\beta}(q)=\min\left\{-\frac{\frac{1}3\log \sum_{j=1}^{2}
p_{a,j}^{q} +\frac{2}3\log \sum_{j=1}^{3} p_{b,j}^{q}}{\frac{1}3\log
r_{a}+\frac{2}3\log r_{b}},-\frac{\frac{2}3\log \sum_{j=1}^{2}
p_{a,j}^{q} +\frac{1}3\log \sum_{j=1}^{3} p_{b,j}^{q}}{\frac{2}3\log
r_{a}+\frac{1}3\log r_{b}}\right\}
$$
and
$$
\overline{\beta}(q)=\max\left\{-\frac{\frac{1}3\log \sum_{j=1}^{2}
p_{a,j}^{q} +\frac{2}3\log \sum_{j=1}^{3} p_{b,j}^{q}}{\frac{1}3\log
r_{a}+\frac{2}3\log r_{b}},-\frac{\frac{2}3\log \sum_{j=1}^{2}
p_{a,j}^{q} +\frac{1}3\log \sum_{j=1}^{3} p_{b,j}^{q}}{\frac{2}3\log
r_{a}+\frac{1}3\log r_{b}}\right\}.
$$
It is obvious that $\underline{\beta}(q)\leq
{\beta}_k(q)\leq\overline{\beta}(q)$ for all $k\geq 2$, the
functions $\underline{\beta}$ and $\overline{\beta}$ are are
strictly decreasing, $\underline{\beta}(1)=\overline{\beta}(1)=0$
and $\overline{\beta}(q)>\underline{\beta}(q)$ for all $q\neq 1$
(see Fig. \ref{Fig1}).

\begin{figure}[h!]\centering
\includegraphics[height=5cm]{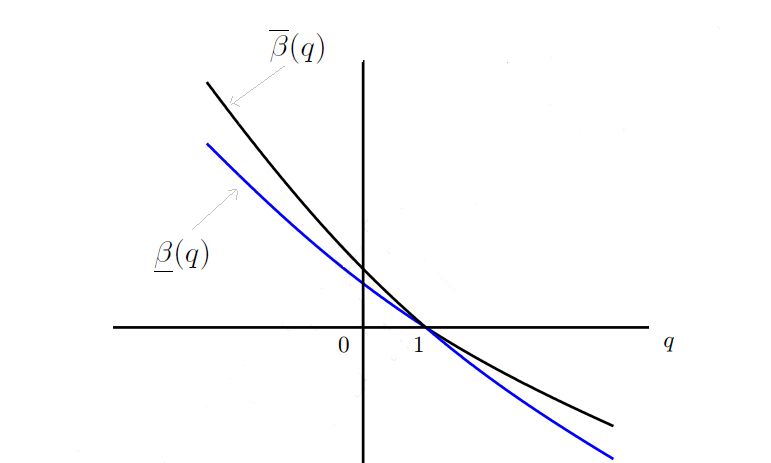}
 \caption{ The multifractal functions $\overline{\beta}$ and $\underline{\beta}$. }\label{Fig1}
\end{figure}

At last, we compute the dimension of the level sets $E_\mu(\alpha)$
(see Fig. \ref{Fig2}).
\begin{theorem}\label{Shen111} Suppose that $E$ is a homogeneous Moran set satisfying
(\texttt{SSC}) and $\mu$ is the Moran measure on $E$. Let $q\in
\mathbb{R}$ and assume that $\underline{\beta}'(q)$ (resp.
$\overline{\beta}'(q)$) exists. Then,

$$
\underline{\dim}_{MB}\big(E_\mu({-\underline{\beta}'(q)})\big)=\underline{\beta}^*(-\underline{\beta}'(q))\;\;\text{provided}
\;\;0<\displaystyle\liminf_{k\to +\infty}\sum_{\sigma\in{D}_{k}}
\mu\left(J_{\sigma}\right)^{q}\left|J_{\sigma}\right|^{\underline{\beta}(q)}<+\infty
$$
and
$$
\overline{\dim}_{MB}\big(E_\mu({-\overline{\beta}'(q)})\big)=\overline{\beta}^*(-\overline{\beta}'(q))\;\;\text{provided}
\;\;0<\displaystyle\limsup_{k\to +\infty}\sum_{\sigma\in{D}_{k}}
\mu\left(J_{\sigma}\right)^{q}\left|J_{\sigma}\right|^{\overline{\beta}(q)}<+\infty.
$$
\end{theorem}

\begin{figure}[h!]\centering
 \includegraphics[height=4.5cm]{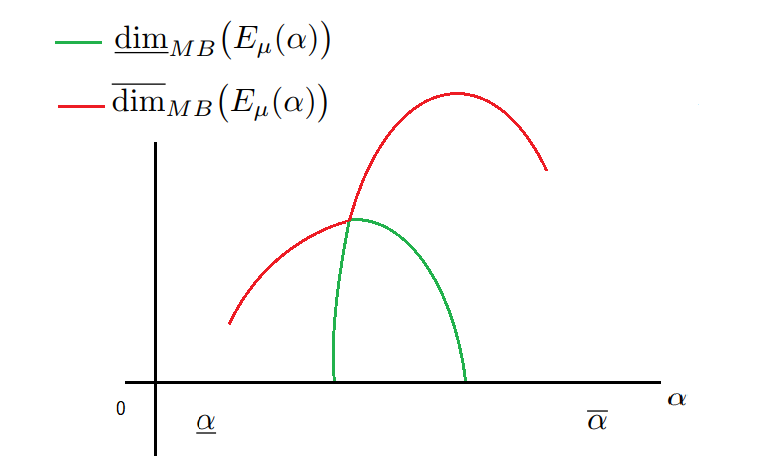}\\
 \includegraphics[height=4.5cm]{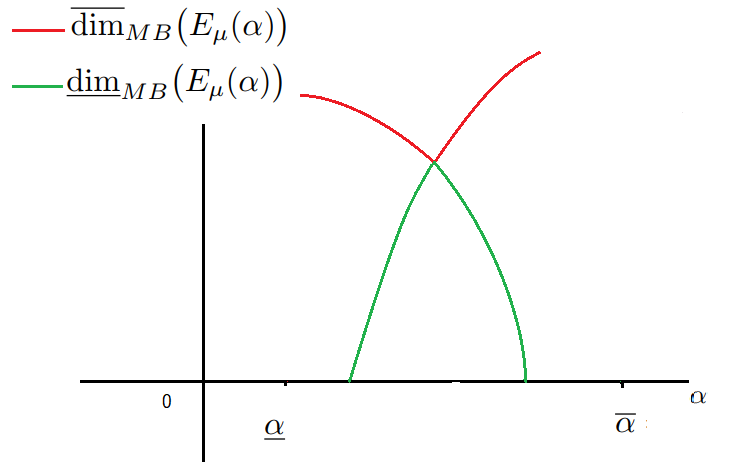}\\
\includegraphics[height=4.5cm]{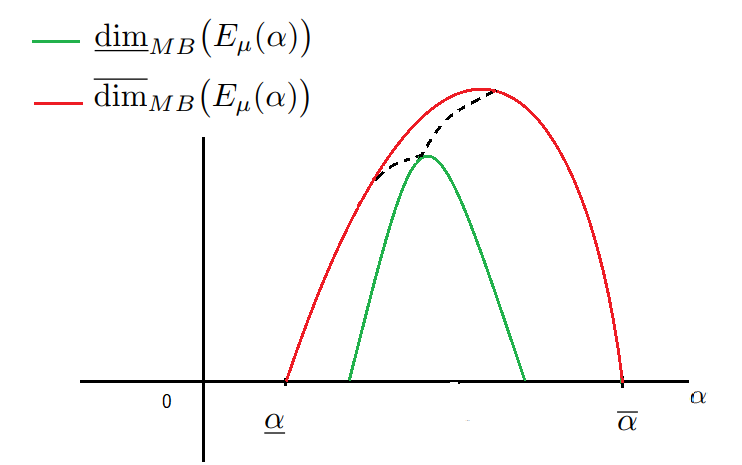}
 \caption{ The Hewitt-Stromberg dimensions of the set
 ${E}_{\mu}(\alpha)$.}\label{Fig2}
\end{figure}

\begin{proof}  For
$q\in\mathbb{R},$ let $\nu$ be a probability measure on $\supp\mu$
such that for any $k\ge1$ and $\sigma_0\in D_k$,
$$\nu\left(J_{\sigma_{0}}\right)=\frac{\mu\left(J_{\sigma_{0}}\right)^{q}\left|J_{\sigma_{0}}\right|^{\underline{\beta}(q)}}
{\displaystyle\sum_{\sigma\in{D}_{k}}\mu\left(J_{\sigma}\right)^{q}\left|J_{\sigma}\right|^{\underline{\beta}(q)}}.$$
Proposition \ref{Prop1} implies that
$$
{\mathsf B}_{\nu}^{1}\big(E_\mu({-\underline{\beta}'(q)})\big)\leq
{\mathsf B}_{\nu}({1})=0.
$$
It follows from Theorem \ref{thinf} that
\begin{eqnarray*}
\underline{\dim}_{MB}\big(E_\mu({-\underline{\beta}'(q)})\big)\leq
-q\underline{\beta}'(q)+\underline{\beta}(q).
\end{eqnarray*}
We can see from the construction of measure $\nu$ that
$\nu\big(E_\mu({-\underline{\beta}'(q)})\big)>0$ (we can see also
\cite{W3}) which implies that
$\pi\big(E_\mu({-\underline{\beta}'(q)})\big)>0.$ We therefore
conclude from Theorem \ref{thess} that
\begin{eqnarray*}
\underline{\dim}_{MB}\big(E_\mu({-\underline{\beta}'(q)})\big)\geq
-q\underline{\beta}'(q)+\underline{\beta}(q)
\end{eqnarray*}
which yields the desired result. The proof of the second statement
is identical to the proof of the first statement and is therefore
omitted.
\end{proof}
\begin{remark} Note that the model of Yuan \cite{YZ} (Huang et al. \cite{HLW})
can be adapted to our general theory which provides a non-trivial
sample with a full description of the refined multifractal analysis
based on multifractal Hewitt-Stromberg measures.

Let $J=[0,1]$, $n_i=2$ and $\mathscr{N}:=\{N_k\}_{k\in\mathbb{N}}$
be an increasing sequence of integers with $N_0=0$ and
$\displaystyle{\lim_{k\to+\infty}\frac{N_{k+1}}{N_{k}}=+\infty}$.
Fix four real numbers $A,B,p,\widetilde{p}$ with $A>B>2$ and
$0<p,\widetilde{p}\leq 1/2$. Now for every $i\in\mathbb{N}$, we
define $c_i$ and $\{P_{i,j}\}_{1\leq j\leq n_i}$ as follows:
$$
c_i=\left\{
\begin{array}{ll}
    \displaystyle\frac{1}{A},\;\text{if}\;\; N_{2k}< i\leq N_{2k+1}, &\\
    \\\\
    \displaystyle\frac{1}{B},\;\text{if}\;\; N_{2k+1}< i\leq N_{2k+2}. &
\end{array}
\right.
$$
and
$$
 P_{i,j}=\left\{
\begin{array}{ll}
    p,\;\text{if}\;\; N_{2k}< i\leq N_{2k+1} \text{ and } j=1, &\\
    \\
    1-p,\;\text{if}\;\; N_{2k}< i\leq N_{2k+1} \text{ and } j=2,
    &\\\\
    \widetilde{p},\;\text{if}\;\; N_{2k+1}< i\leq N_{2k+2}\text{ and } j=1,
    &\\\\
    1-\widetilde{p},\;\text{if}\;\; N_{2k+1}< i\leq N_{2k+2}\text{ and } j=2.
\end{array} \right.
$$
Define now the following functions
\begin{eqnarray*}
    \beta_1:&\mathbb{R}\to& \mathbb{R}\\
    &q\mapsto&\frac{\log(p^q+(1-p)^q)}{\log A},
\end{eqnarray*}
and
\begin{eqnarray*}
    \beta_2:&\mathbb{R}\to& \mathbb{R}\\
    &q\mapsto&\frac{\log(\widetilde{p}^q+(1-\widetilde{p})^q)}{\log
    B}.
\end{eqnarray*}
We can conclude that
$${\mathsf
b}_\mu(q)=\min\Big\{\beta_1(q),\beta_2(q)\Big\}$$ and  $${\mathsf
B}_\mu(q)=\max\Big\{\beta_1(q),\beta_2(q)\Big\}.$$ If $-\frac{\log
(1-\widetilde{p})}{\log B}<-\frac{\log p}{\log A}$, then for all
$\alpha\in \left[-\frac{\log (1-\widetilde{p})}{\log
B},\min\{-\frac{\log p}{\log A},-\frac{\log \widetilde{p}}{\log
B}\}\right]$
$$\underline{\dim}_{MB}\big(E_\mu(\alpha)\big)={\mathsf
b}_\mu^*(\alpha),$$ and for $\alpha\in\Big\{{\mathsf
B}_\mu'(q):q\in\mathbb{R} \text{ and } {\mathsf B}_\mu \text{ is
differentiable at } q\Big\}$ we have
$$\overline{\dim}_{MB}\big(E_\mu(\alpha)\big)={\mathsf
B}_\mu^*(\alpha).$$ Even if $p=\widetilde{p},$
$$
\liminf_{k\to+\infty}\frac{\sharp\left\{1\leq i\leq k;\;
c_i=\frac{1}{A}\right\}}{k}=0$$ and
$$\limsup_{k\to+\infty}\frac{\sharp\left\{1\leq i\leq k;\;
c_i=\frac{1}{A}\right\}}{k}=1,
$$
then the phenomena are the same, i.e., for $\alpha\in
\left[-\frac{\log (1-{p})}{\log B},\;-\frac{\log p}{\log A}\right]$,
we have
$$\underline{\dim}_{MB}\big(E_\mu(\alpha)\big)={\mathsf
b}_\mu^*(\alpha),$$ and for $\alpha\in\left[-\frac{\log
(1-{p})}{\log B},\;-\frac{\log p}{\log
A}\right]\setminus\left(\frac{-p \log p - (1 - p) \log(1 - p)}{\log
A},\; \frac{-p \log p - (1 - p) \log(1 - p)}{\log B}\right)$, we
have
$$\overline{\dim}_{MB}\big(E_\mu(\alpha)\big)={\mathsf
B}_\mu^*(\alpha).$$ This implies that the results of Theorem
\ref{Shen111} hold for some
$(\alpha,\beta)\neq\left(-\underline{\beta}'(q),
-\overline{\beta}'(q)\right).$
\end{remark}

\bigskip\bigskip In the following, we give some examples of a
measure for which the lower and upper multifractal Hewitt-Stromberg
functions are different (see Fig. \ref{Fig3}) and the
Hewitt-Stromberg dimensions of the level sets of the local
H\"{o}lder exponent $E_\mu(\alpha)$ are given by the Legendre
transform respectively of lower and upper multifractal
Hewitt-Stromberg functions (see Fig. \ref{Fig4}). In particular, we
prove that our multifractal formalism \cite[Theorem 8]{NB2} holds
for these measures.

\subsection{Example 3}
Take $0<p<\hat{p}\leq 1/2$ and a sequence of integers
$$
1=t_0<t_1<\cdots<t_n<\cdots, \;\; \text{such that}
\;\;\lim_{n\to+\infty}\frac{t_{n+1}}{t_n}=+\infty. $$ The measure
$\mu$ assigned to the diadic interval of the n-th generation
$I_{\varepsilon_1\varepsilon_2\cdots\varepsilon_n}$ is
$$
\mu\big(I_{\varepsilon_1\varepsilon_2\cdots\varepsilon_n}\big)=\prod_{j=1}^n
\varpi_j,
$$
where
$$
\left\{
  \begin{array}{ll}
  \text{if}\;\; t_{2k-1}\leq j<t_{2k}\;\;\text{for some}\;\;k,\;
\varpi_j=p\;\text{if}\;\varepsilon_j=0,\; \varpi_j=1-p\;\;\text{otherwise}   , &\\
\\
  \text{if}\;\; t_{2k}\leq j<t_{2k+1}\;\;\text{for some}\;\;k,\;
 \varpi_j=\hat{p}\;\text{if}\;\varepsilon_j=0,\; \varpi_j=1-\hat{p}\;\;\text{otherwise.} &
  \end{array}
\right.
$$
We observe that
$$
\sum_{{\varepsilon_1\varepsilon_2\cdots\varepsilon_{n}}}\mu\big(I_{\varepsilon_1\varepsilon_2\cdots\varepsilon_{n}}\big)^q=\big(p^q+(1-p)^q\big)^{N_n}
\big(\hat{p}^q+(1-\hat{p})^q\big)^{n-N_n},
$$
where  $N_n$ is the number of integers $j\leq n$ such that
$\varpi_j=p$. It is clear that  $\liminf_{n\to+\infty} (N_n/n)=0$
and $\limsup_{n\to+\infty} (N_n/n)=1$. Now, for $q\in \mathbb{R},$
we define
$$
\underline{\tau}(q)=\log_2\big(p^q+(1-p)^q\big)\quad\text{and}\quad
\overline{\tau}(q)=\log_2\big(\hat{p}^q+(1-\hat{p})^q\big).
$$
Then
$$
\liminf_{n\to+\infty}\frac{1}n
\log\sum_{{\varepsilon_1\varepsilon_2\cdots\varepsilon_{n}}}
\mu\big(I_{\varepsilon_1\varepsilon_2\cdots\varepsilon_{n}}\big)^q=\min\big\{\underline{\tau}(q),\overline{\tau}(q)\big\}
$$
and
$$
\limsup_{n\to+\infty}\frac{1}n
\log\sum_{{\varepsilon_1\varepsilon_2\cdots\varepsilon_{n}}}
\mu\big(I_{\varepsilon_1\varepsilon_2\cdots\varepsilon_{n}}\big)^q=\max\big\{\underline{\tau}(q),\overline{\tau}(q)\big\}.
$$
It results from \cite[Proposition 2]{NB2} and  \cite{BJ, BBH} that
$$
{\mathsf
b}_{\mu}(q)=\min\big\{\underline{\tau}(q),\overline{\tau}(q)\big\}\quad\text{and}\quad{\mathsf
B}_{\mu}(q)={\mathsf  \Delta
}_{\mu}(q)=\max\big\{\underline{\tau}(q),\overline{\tau}(q)\big\}
$$
This gives
$$
\left\{
  \begin{array}{ll}
  {\mathsf b}_{\mu}(q)=\underline{\tau}(q)<\overline{\tau}(q)={\mathsf
B}_{\mu}(q)={\mathsf  \Delta }_{\mu}(q), \;\; \text{for}\;\; 0<q<1,&
\\   \\
{\mathsf b}_{\mu}(q)=\overline{\tau}(q)<\underline{\tau}(q)={\mathsf
B}_{\mu}(q)={\mathsf  \Delta }_{\mu}(q), \;\; \text{for}\;\;
q<0\;\text{or}\; q>1
  \end{array}
\right.
$$
and
$$
{\mathsf b}_{\mu}(q)={\mathsf B}_{\mu}(q)={\mathsf  \Delta
}_{\mu}(q), \;\; \text{for}\;\; q\in\{0,1\}.
$$

\begin{figure}[h!]\centering
 \includegraphics[height=4.5cm]{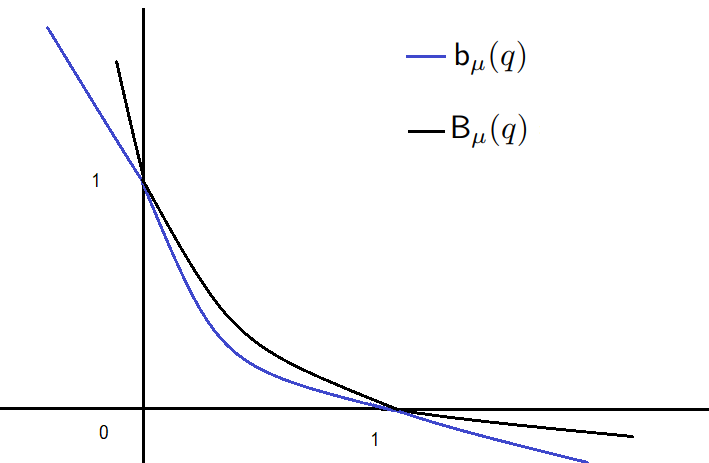}
 \caption{ The relation between the graphs of ${\mathsf b}_{\mu}$ and ${\mathsf B}_{\mu}$.}\label{Fig3}
\end{figure}

Given $0 < p, \hat{p} < 1$, define the mixed entropy function
$$
h(\hat{p}, p) := -\hat{p}\log_2 p - (1 - \hat{p}) \log_2(1 - p).
$$
Then
$$
-\underline{\tau}'(+\infty)=h(0,p)\leq-\underline{\tau}'(1)=h(p,p)\leq-\underline{\tau}'(0)=h(1/2,p)\leq
h(1,p)=-\underline{\tau}'(-\infty)
$$
and
$$
-\overline{\tau}'(+\infty)=h(0,\hat{p})\leq-\overline{\tau}'(1)=h(\hat{p},\hat{p})\leq-\overline{\tau}'(0)=h(1/2,\hat{p})
\leq h(1,\hat{p})=-\overline{\tau}'(-\infty).
$$
Now, we have the following result,
\begin{theorem}\label{th5} Assume that $\alpha\in
\Big(-\overline{\tau}'(+\infty),\;-\overline{\tau}'(-\infty)\Big)$.
\begin{enumerate}
\item For $\alpha\notin
\big[-\mathsf{b}_{\mu_+}'(0),-\mathsf{b}_{\mu_-}'(0)\big]\bigcup
\big[-\mathsf{b}_{\mu_+}'(1),-\mathsf{b}_{\mu_-}'(1)\big]$, we have
$$
{\dim}_{H}\big(E_\mu(\alpha)\big)=\underline{\dim}_{MB}\big(E_\mu(\alpha)\big)={\mathsf
b}_{\mu}^*(\alpha).
$$
\\
\item  For $\alpha\notin
\big[-\mathsf{B}_{\mu_+}'(0),-\mathsf{B}_{\mu_-}'(0)\big]\bigcup
\big[-\mathsf{B}_{\mu_+}'(1),-\mathsf{B}_{\mu_-}'(1)\big]$, we have
$$
{\dim}_{P}\big(E_\mu(\alpha)\big)=\overline{\dim}_{MB}\big(E_\mu(\alpha)\big)={\mathsf
B}_{\mu}^*(\alpha).
$$
\end{enumerate}
\end{theorem}

\begin{figure}[h!]\centering
 \includegraphics[height=5.5cm]{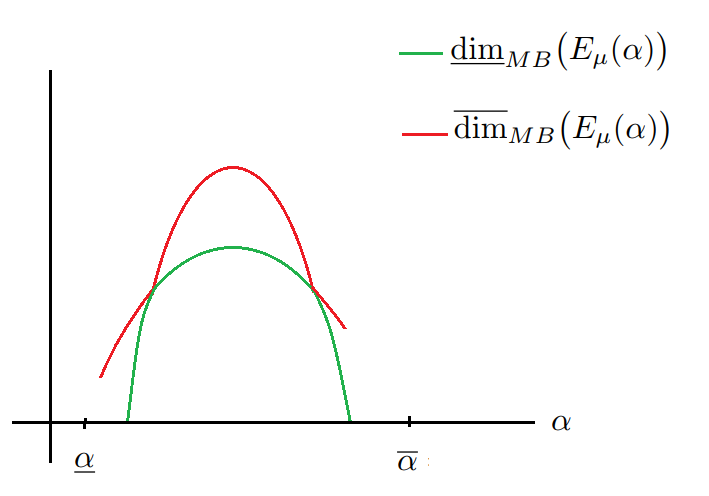}
 \caption{ The Hewitt-Stromberg dimensions of the set
 ${E}_{\mu}(\alpha)$.}\label{Fig4}
\end{figure}

\begin{proof} We can construct a new probability measure $\nu$ on the diadic
interval of the n-th generation just as $\mu$, but replacing $(p,
\hat{p})$ with $(s, \hat{s})$ such that
$$
s \log p + (1 - s)\log(1 - p) = \hat{s} \log \hat{p} + (1 - \hat{s})
\log(1 - \hat{p}),
$$

$$
\log\frac{1-p}{1 -
\hat{p}}<s\log\frac{1-p}{p}<\log\frac{1-p}{\hat{p}}
$$
and
$$
\alpha= -s \log_2 p - (1 - s) \log_2(1 - p) = s \log_2 \frac{1 - p}{
p} - \log_2(1 - p).
$$
From Lemma \ref{Lemma}, we can define the function $f$ by
\begin{eqnarray*}
f(t)&=&\limsup_{r\to 0}\frac{1}{-\log r}\log\left(\sup
\left\{\displaystyle \sum_i
\mu(B(x_i,r))^t\nu(B(x_i,r))\;\Big|\;\Big(B(x_i,r)\Big)_i\;\text{is
a packing of}\; K\right\}\right).
\end{eqnarray*}
Then, it is easy to compute
\begin{eqnarray*}
f(t) &=&\limsup_{n\to
+\infty}\frac{1}{n}\log_2\sum_{{\varepsilon_1\varepsilon_2\cdots\varepsilon_{n}}}
\mu\big(I_{\varepsilon_1\varepsilon_2\cdots\varepsilon_{n}}\big)^t
\nu\big(I_{\varepsilon_1\varepsilon_2\cdots\varepsilon_{n}}\big)\\
&=& \limsup_{n\to +\infty}\frac{1}{n}\log_2 \left(\Big(p^tr + (1 -
p)^t(1- r)\Big)^{N_n} \Big(\hat{p}^t\hat{s} + (1 - \hat{p})^t(1 -\hat{s})\Big)^{n-N_n}\right)\\
&=&\limsup_{n\to +\infty}\left(\frac{N_n}{n}\log_2 \Big(p^tr + (1 -
p)^t(1-
r)\Big)+\left(1-\frac{N_n}{n}\right)\log_2\Big(\hat{p}^t\hat{s} + (1
- \hat{p})^t(1 -\hat{s})\Big)\right)\\
&=& \log_2 \max\Big\{(p^tr + (1 - p)^t(1 - r)),\; (\hat{p}^t\hat{s}
+ (1 - \hat{p})^t(1 - \hat{s}))\Big\}.
\end{eqnarray*}
It is clear that $f(0)=0$, and the method of choosing $(s, \hat{s})$
insures that$f'(0)$ exists and is equal to $-\alpha$.

\bigskip
Now we can estimate the bounds of the dimensions of the level sets.
Given $0 < s < 1$ and define the entropy function
$$
H(s):=-s \log_2 s + (s - 1)\log_2(1 - s).
$$
The strong law of large numbers shows that
$$
\liminf_{n\to +\infty}\frac{\log_2 \nu\big(B(x, 2^{-n})\big)}{- n}=
\min\Big\{H(s),\; H(\hat{s})\Big\}
$$
and
$$
\limsup_{n\to +\infty}\frac{\log_2 \nu\big(B(x, 2^{-n})\big)}{- n}=
\max\Big\{H(s),\; H(\hat{s})\Big\}
$$
for $\nu$-almost every $x$. So it deduces from Theorem \ref{thess}
that
$$
\underline{\dim}_{MB}\big(E_\mu(\alpha)\big)\geq\min\Big\{H(s),\;
H(\hat{s})\Big\}
$$
and
$$
\overline{\dim}_{MB}\big(E_\mu(\alpha)\big)\geq\max\Big\{H(s),\;
H(\hat{s})\Big\}.
$$
To compute $H(s)$ and $H(\hat{s})$, set
$$
q=\frac{\log\frac{1-s}{s}}{\log\frac{1-p}{p}}
$$
then
\begin{eqnarray*}
\underline{\tau}'(q)&=&\frac{p^q \log_2
p+(1-p)^q\log_2(1-p)}{p^q+(1-p)^q}\\&=&\frac{ \log_2
p+\left(\frac{(1-p)}{p}\right)^q\log_2(1-p)}{1+\left(\frac{(1-p)}{p}\right)^q}\\&=&\frac{
\log_2
p+\left(\frac{(1-s)}{s}\right)\log_2(1-p)}{1+\left(\frac{(1-s)}{s}\right)}\\&=&s\log_2
p+(1-s)\log_2(1-p)\\ &=&f'(0)=-\alpha.
\end{eqnarray*}
Which implies that
\begin{eqnarray*}
\underline{\tau}(q)-q\underline{\tau}'(q)&=&\log_2\big(p^q+(1-p)^q\big) -q\underline{\tau}'(q)\\
&=&\log_2p^q~\left(1+\left(\frac{(1-p)}{p}\right)^q\right)
-q\underline{\tau}'(q)\\
&=& q \log_2p + \log_2
\left(1+\left(\frac{(1-p)}{p}\right)^q\right)-q\underline{\tau}'(q)\\
&=& q \log_2p + \log_2 \left(1+\frac{(1-s)}{s}\right)-q\underline{\tau}'(q)\\
&=&q \log_2p - \log_2s -q\Big(s\log_2 p+(1-s)\log_2(1-p)\Big)\\
&=&- \log_2s+q\Big((1-s)\log_2 p-(1-s)\log_2(1-p)\Big)\\
&=&- \log_2s -(1-s)\log_2 \frac{(1-s)}{s}\\
&=& - s\log_2s -(1-s)\log_2 (1-s)\\
&=& H(s).
\end{eqnarray*}
Also, set
$$
\hat{q}=\frac{\log\frac{1-\hat{s}}{\hat{s}}}{\log\frac{1-\hat{p}}{\hat{p}}}
$$
with the very same arguments, we have
$$
\overline{\tau}'(\hat{q})=-\alpha\quad\text{and}\quad\overline{\tau}(\hat{q})-\hat{q}\overline{\tau}'(\hat{q})=
H(\hat{s}).
$$
Thus
$$
H(s)=\underline{\tau}(q)-q\underline{\tau}'(q)=\underline{\tau}^*(-\underline{\tau}'(q))
$$
and
$$
H(\hat{s})=\overline{\tau}(\hat{q})-\hat{q}\overline{\tau}'(\hat{q})=
\overline{\tau}^*(-\overline{\tau}'(\hat{q})),
$$
which give the lower bounds of the dimensions of the level sets,
i.e.,
$$
\underline{\dim}_{MB}\big(E_\mu(\alpha)\big)\geq\min\Big\{\underline{\tau}^*(-\underline{\tau}'(q)),\;
\overline{\tau}^*(-\overline{\tau}'(\hat{q}))\Big\}
$$
and
$$
\overline{\dim}_{MB}\big(E_\mu(\alpha)\big)\geq\max\Big\{\underline{\tau}^*(-\underline{\tau}'(q)),\;
\overline{\tau}^*(-\overline{\tau}'(\hat{q}))\Big\}.
$$
But we have also the opposite inequalities:\\ \\
In order to have $\underline{\tau}(q)={\mathsf b}_{\mu}(q)$, we must
have $0 < q < 1$, which means
$$
-\underline{\tau}'(1)=h(p,p)<\alpha<h(1/2,p)=-\underline{\tau}'(0).
$$
In order to have $\overline{\tau}(\hat{q})={\mathsf
b}_{\mu}(\hat{q})$, we must have $ \hat{q}< 0 $ or $\hat{q} > 1$,
which means
$$
\alpha>h(1/2,\hat{p})=-\overline{\tau}'(0)\quad\text{or}\quad
\alpha<h(\hat{p},\hat{p})=-\overline{\tau}'(1).
$$
In order to have $\overline{\tau}(\hat{q})={\mathsf
B}_{\mu}(\hat{q})$, we must have $0 < \hat{q} < 1$, which means
$$
-\overline{\tau}'(1)=h(\hat{p},\hat{p})<\alpha<h(1/2,\hat{p})=-\overline{\tau}'(0).
$$
In order to have ${\underline{\tau}}({q})={\mathsf B}_{\mu}({q})$,
we must have $ {q}< 0 $ or ${q} > 1$, which means
$$
\alpha>h(1/2,{p})=-\underline{\tau}'(0)\quad\text{or}\quad
\alpha<h({p},{p})=-\underline{\tau}'(1).
$$
Now, put
$$
I=\Big(-\overline{\tau}'(+\infty),\;-\overline{\tau}'(-\infty)\Big)\setminus
\big[-\underline{\tau}'(0),-\overline{\tau}'(0)\big]\bigcup
\big[-\overline{\tau}'(1),-\underline{\tau}'(1)\big]
$$
and
$$
J=\Big(-\overline{\tau}'(+\infty),\;-\overline{\tau}'(-\infty)\Big)\setminus\big[-\overline{\tau}'(0),-\underline{\tau}'(0)\big]\bigcup
\big[-\underline{\tau}'(1),-\overline{\tau}'(1)\big].
$$
It is easy to verify that $I,J\subseteq (\underline{\alpha},
\;\overline{\alpha})$. Finally, it follows from Theorem \ref{Jcole}
that
$$
\underline{\dim}_{MB} \big(E_{\mu}(\alpha)\big)\leq { \mathsf
b}_{\mu}^*(\alpha)\quad\text{and}\quad \overline{\dim}_{MB}
\big(E_{\mu}(\alpha)\big)\leq {\mathsf B}_{\mu}^*(\alpha),
$$
which yields the desired result.
\end{proof}
\subsection{Example 4} For ${\mathscr X} = \{0, 1, 2, 3\}$, we consider ${\mathscr X}^* = \bigcup_{n\geq0} {\mathscr X}^n$, the set of all finite words
on the $4$-letter alphabet ${\mathscr X}$. Let $w
=\varepsilon_1\cdots \varepsilon_n$ and $v = \varepsilon_{n+1}\cdots
\varepsilon_{n+m}$, denote by $wv$ the word $\varepsilon_1\cdots
\varepsilon_{n+m}$. With this operation, ${\mathscr X}^*$ is a
monoid whose identity element is the empty word $\epsilon$. If a
word $v$ is a prefix of the word $w$, we write $v\prec w$. This
defines an order on ${\mathscr X}^*$ and endowed with this order,
${\mathscr X}^*$ becomes a tree whose root is $\epsilon$. At last,
the length of a word $w$ is denoted by $|w|$. If $w$ and $v$ are two
words, $w\wedge v$ stands for their largest common prefix. It is
well known that the function $d:=dist(w, v) = 4^{-|w\wedge v|}$
defines an ultra-metric distance on ${\mathscr X}^*$. The completion
of $({\mathscr X}^*, d)$ is a compact space which is the disjoint
union of ${\mathscr X}^*$ and $\partial{\mathscr X}^*$, whose
elements can be viewed as infinite words. Each finite word $w\in
{\mathscr X}^*$ defines a cylinder $[w]=\{x\in\partial{\mathscr
X}^*\;|\;w \wedge x\}$, which can also be viewed as a ball. Let
$a_i, b_j\in (0, 1),$ $i,j\in\{1,2,3,4\}$ satisfying
$$
\sum_{i=1}^4 a_i=\sum_{j=1}^4b_j=1
$$
and $(t_k)$ be a sequence of integers such that
$$
t_1=1,\;\; t_k<t_{k+1}\quad\text{and}\quad
\lim_{k\to+\infty}\frac{t_{k+1}}{t_{k}}=+\infty.
$$
We define the measure $\mu$ on $\partial{\mathscr X}^*$ such that
for every cylinder $[\varepsilon_1\varepsilon_2 \cdots
\varepsilon_n]$, one has
$$
\mu([\varepsilon_1\varepsilon_2 \cdots \varepsilon_n])=\prod_{j=1}^n
p_j,
$$
where
$$
\left\{
  \begin{array}{ll}
  \text{if}\;\; t_{2k-1}\leq j<t_{2k}\;\;\text{for some}\;\;k,\;
p_j=a_{\varepsilon_i+1},& \\
\\
  \text{if}\;\; t_{2k}\leq j<t_{2k+1}\;\;\text{for some}\;\;k,\;
 p_j=b_{\varepsilon_i+1}.&
  \end{array}
\right.
$$
Then
$$
{\mathsf
b}_{\mu}(q)=\inf\left\{\log_4\Big(a_1^q+a_2^q+a_3^q+a_4^q\Big),\;\log_4\Big(b_1^q+b_2^q+b_3^q+b_4^q\Big)\right\}
$$
and
$$
{\mathsf
B}_{\mu}(q)=\sup\left\{\log_4\Big(a_1^q+a_2^q+a_3^q+a_4^q\Big),\;\log_4\Big(b_1^q+b_2^q+b_3^q+b_4^q\Big)\right\}.
$$
The functions ${\mathsf b}_{\mu}$ and ${\mathsf B}_{\mu}$ are
analytic and their graphs differ except at two points where they are
tangent, with ${\mathsf b}_{\mu}(0) = {\mathsf B}_{\mu}(0)$,
${\mathsf b}_{\mu}(1) = {\mathsf B}_{\mu}(1)$, and ${\mathsf
B}_{\mu}(q) > {\mathsf b}_{\mu}(q)$ for all $q \neq 0, 1$ (see
Figure \ref{Fig3}). Moreover ${\mathsf b}_{\mu}$ and ${\mathsf
B}_{\mu}$ are convex and ${\mathsf B}_{\mu}'(\mathbb{R})$ and
${\mathsf b}_{\mu}'(\mathbb{R})$ both are intervals of positive
length (for more detail see \cite{SH1}).

Now, we suppose that $a_1 < b_1$. Then by construction of the
measure $\mu$, the graph of
$\log_4\left(a_1^q+a_2^q+a_3^q+a_4^q\right)$ is always on top of the
graph of $\log_4\left(b_1^q+b_2^q+b_3^q+b_4^q\right)$. So, we get
$$
{\mathsf b}_{\mu}(q)=\log_4\Big(b_1^q+b_2^q+b_3^q+b_4^q\Big)$$and
$${\mathsf B}_{\mu}(q)=\log_4\Big(a_1^q+a_2^q+a_3^q+a_4^q\Big).
$$
\begin{theorem} We assume that
$\alpha\in\big(-\log_4b_4,\;-\log_4b_1\big)$. Then
$$
{\dim}_{H}\big(E_\mu(\alpha)\big)=\underline{\dim}_{MB}\big(E_\mu(\alpha)\big)={\mathsf
b}_{\mu}^*(\alpha)
$$
and
$$
{\dim}_{P}\big(E_\mu(\alpha)\big)=\overline{\dim}_{MB}\big(E_\mu(\alpha)\big)={\mathsf
B}_{\mu}^*(\alpha).
$$
\end{theorem}
\begin{proof}
The proof is very similar to the one of Theorem \ref{th5}. We can
see also \cite{SH1} for the estimates of the Hausdorff and packing
dimensions of the set $E_\mu(\alpha)$.
\end{proof}
\begin{remark}  Here $\mathscr{K}(\mathbb{R}^n)$ denotes
the family of non-empty compact subsets of $\mathbb{R}^n$ equipped
with the Hausdorff metric, and $\mathscr{P}(\mathbb{R}^n)$ denotes
the family of Radon measures on $\mathbb{R}^n$ equipped with the
weak topology. The study of the descriptive set-theoretic complexity
of the maps
\begin{eqnarray*}
\mathscr{K}(\mathbb{R}^n)&\times& \mathscr{P}(\mathbb{R}^n)\times
\mathbb{R} \longrightarrow [-\infty,\; +\infty]\;:\; (K, \mu,
q)\mapsto \mathsf{H}_\mu^{q,t}(K),\\\\
\mathscr{K}(\mathbb{R}^n)&\times& \mathscr{P}(\mathbb{R}^n)\times
\mathbb{R} \longrightarrow [-\infty,\; +\infty]\;:\; (K, \mu,
q)\mapsto \mathsf{P}_\mu^{q,t}(K),\\\\
\mathscr{K}(\mathbb{R}^n)&\times& \mathscr{P}(\mathbb{R}^n)\times
\mathbb{R} \longrightarrow [-\infty,\; +\infty]\;:\; (K, \mu,
q)\mapsto \mathsf{b}_\mu^{q}(K),\\\\
\mathscr{K}(\mathbb{R}^n)&\times& \mathscr{P}(\mathbb{R}^n)\times
\mathbb{R} \longrightarrow [-\infty,\; +\infty]\;:\; (K, \mu,
q)\mapsto \mathsf{B}_\mu^{q}(K),\\\\
\mathscr{K}(\mathbb{R}^n)&\times& \mathscr{P}(\mathbb{R}^n)\times
\mathbb{R} \longrightarrow [-\infty,\; +\infty]\;:\; (K, \mu,
q)\mapsto \mathsf{\Delta}_\mu^{q}(K)
\end{eqnarray*}
and the multifractal structure of product measures and dimensions
(note that Edgar and Zindulka in \cite{Ed,Zi} studied the structure
of the Hewitt-Stromberg measures and dimensions on cartesian
products in the case $q=0$) will be achieved in further works.
\end{remark}
                                                                \section*{Acknowledgments}

\noindent The author would like to thank Professors {\it Fathi Ben
Nasr} and {\it Imen Bhouri} for highlighting the problem of the
refined multifractal analysis. The author is greatly indebted to
{\it Zhihui Yuan} for carefully reading the first version of this
paper and giving elaborate comments and valuable suggestions so that
the presentation can be greatly improved, especially for producing
some above figures. The author would like also to thank Professor
{\it Jinjun Li} for pointing out that the upper
(fractal/multifractal) Hewitt-Stromberg function may not be a metric
outer measure, and  Professor {\it Lars Olsen} for the
counterexample. And, the author gratefully thanks the referees for
their careful reading and valuable suggestions. This work was
supported by Analysis, Probability $\&$ Fractals Laboratory (No:
LR18ES17).


\end{document}